\documentclass[a4paper,twoside,10pt,reqno]{amsart}
\usepackage[final,kerning=true,spacing=true,factor=1100,stretch=10,shrink=10]{microtype}

\usepackage[a4paper,left=3cm,right=3cm, top=3cm, bottom=3cm]{geometry}
\usepackage[latin1]{inputenc}
\usepackage{mathrsfs}
\usepackage{enumerate}
\usepackage{graphicx}
\usepackage{epstopdf}
\usepackage{amsmath}
\usepackage{amsthm}
\usepackage{algorithm,algpseudocode}
\usepackage{amssymb}
\usepackage{hyperref}
\usepackage{tikz}
\usepackage{stmaryrd}
\usepackage{subfigure}
\usepackage{float}
\usepackage{bigints}
\usepackage{cite}
\usepackage{color}
\usepackage[abs]{overpic}
\usepackage[font=footnotesize,labelfont=bf]{caption}
\usepackage{cases}
\usepackage[author={Lorenzo}]{pdfcomment}
\usepackage{comment}
\usepackage{algorithm,algpseudocode}
\usepackage{xcolor}

\restylefloat{table}
\theoremstyle{plain}
\newtheorem{thm}{Theorem}[section]
\newtheorem{cor}[thm]{Corollary}
\newtheorem{lem}[thm]{Lemma}
\newtheorem{prop}[thm]{Proposition}
\newtheorem{definition}[thm]{Definition}
\newtheorem{remark}[thm]{Remark}

\newcommand{\E}{K}
\newcommand{\Ehat}{\widehat \E}
\newcommand{\Fhat}{\widehat \F}
\newcommand{\Etilde}{\widetilde\E}
\newcommand{\Em}{\E^-}
\newcommand{\Ep}{\E^+}
\newcommand{\F}{F}
\newcommand{\FB}{\F^B}
\newcommand{\h}{h}
\newcommand{\hE}{\h_\E}

\newcommand{\hF}{\h_\F}
\newcommand{\p}{p}
\newcommand{\pE}{\p_\E}
\newcommand{\pF}{\p_\F}
\newcommand{\f}{f}
\newcommand{\nbf}{\mathbf n}
\newcommand{\nbfOmega}{\nbf_\Omega}

\newcommand{\D}{D}
\newcommand{\Dn}{\D_n}
\newcommand{\B}{B}
\newcommand{\Bn}{\B_n}
\newcommand{\V}{V}
\newcommand{\Vn}{V_n}
\newcommand{\psibf}{\boldsymbol{\psi}}
\newcommand{\wbf}{\mathbf w}
\newcommand{\vbf}{\mathbf v}
\newcommand{\taun}{\mathcal T_n}
\newcommand{\En}{\mathcal E_n}
\newcommand{\EnI}{\En^I}
\newcommand{\EnB}{\En^B}

\newcommand{\nbfF}{\nbf_\F}
\newcommand{\Lcal}{\mathcal L}
\newcommand{\un}{u_n}
\newcommand{\uc}{u_c}
\newcommand{\vn}{v_n}

\newcommand{\nablan}{\nabla_n}
\newcommand{\rhobold}{\boldsymbol \rho}
\newcommand{\Psibold}{\boldsymbol \Psi}

\newcommand{\curl}{\text{curl}}
\newcommand{\curlbf}{\textbf{curl}}
\newcommand{\etaE}{\eta_\E}
\newcommand{\EE}{\mathcal E^\E}
\newcommand{\tbf}{\mathbf t}

\newcommand{\I}{I}
\newcommand{\Ibf}{\textbf {I}}
\newcommand{\ec}{e_c}
\newcommand{\etan}{\eta_n}
\newcommand{\Deltan}{\Delta_n}
\newcommand{\bE}{b_\E}

\newcommand{\bEtilde}{b_{\widetilde \E}}
\newcommand{\bF}{b_\F}
\newcommand{\Eit}{\textit{E}}
\newcommand{\bell}{b_\ell}
\newcommand{\qp}{q_\p}
\newcommand{\qpbf}{q_{\pbf}}
\newcommand{\Ihat}{\widehat I}
\newcommand{\dx}{dx}
\newcommand{\dy}{dy}
\newcommand{\psiF}{\psi_F}
\newcommand{\cetaEtw}{c_{\eta_2,\eta_3}}
\newcommand{\cetaEtwsq}{c_{\eta_2,\eta_3}^2}

\newcommand{\omegaE}{\omega_\E}
\newcommand{\hbf}{\mathbf \h}

\newcommand{\pbf}{\mathbf \p}

\newcommand{\xbf}{\mathbf x}

\newcommand{\taubold}{\boldsymbol{\tau}}
\newcommand{\ssigmabold}{\boldsymbol{\Sigma}}
\newcommand{\sigmabold}{\boldsymbol{\sigma}}
\newcommand{\alphabold}{\boldsymbol{\alpha}}

\newcommand{\nonconferr}{e_{nc}}
\newcommand{\conferr}{e_{c}}
\newcommand{\gD}{g_1}
\newcommand{\gN}{g_2}

\usepackage[normalem]{ulem}

\newcommand{\vertiii}[1]{{\left\vert\kern-0.25ex\left\vert\kern-0.25ex\left\vert #1  \right\vert\kern-0.25ex\right\vert\kern-0.25ex\right\vert}}

\numberwithin{equation}{section}

\author{Zhaonan Dong \and
Lorenzo Mascotto \and
Oliver J. Sutton}

\address[Z. Dong]{ Inria, 2 rue Simone Iff, 75589 Paris, France and
 CERMICS, Ecole des Ponts, 77455 Marne-la-Vall\'{e}e 2, France }
 \email{zhaonan.dong@inria.fr}

\address[L. Mascotto]{Fakult\"at f\"ur Mathematik, Universit\"at Wien, 1090 Vienna, Austria}
\email{lorenzo.mascotto@univie.ac.at}
\address[O. J. Sutton]{School of Mathematical Sciences, University of Nottingham, NG7 2QL United Kingdom}
\email{Oliver.Sutton@nottingham.ac.uk}

\date{}

\title[A posteriori error estimates for $hp$-dG schemes for biharmonic problems]{Residual-based a posteriori error estimates for \\ $\boldsymbol{\lowercase{hp}}$-discontinuous Galerkin discretisations of \\the biharmonic problem}

\begin{document}
%%%%%%%%%%%%%%%%%%%%%%%%%%%%%%%%%%%%%
	
\begin{abstract}
\noindent
We introduce a residual-based \emph{a posteriori} error estimator for a novel $\h\p$-version interior penalty discontinuous Galerkin method for the biharmonic problem in two and three dimensions. 
We prove that the error estimate provides an upper bound and a local lower bound on the error, and that the lower bound is robust to the local mesh size but not the local polynomial degree.
The suboptimality in terms of the polynomial degree is fully explicit and grows at most algebraically.
Our analysis does not require the existence of a $\mathcal{C}^1$-conforming piecewise polynomial space and is instead based on an elliptic reconstruction of the discrete solution
to the~$H^2$ space and a generalised Helmholtz decomposition of the error.
This is the first $hp$-version error estimator for the biharmonic problem in two and three dimensions.
The practical behaviour of the estimator is investigated through numerical examples in two and three dimensions.
\medskip
		
\noindent
\textbf{AMS subject classification}: 65N12, 65N30, 65N50.
		
\medskip
		
\noindent
\textbf{Keywords}: discontinuous Galerkin methods; adaptivity; $\h\p$-Galerkin methods; polynomial inverse estimates; fourth order PDEs; a posteriori error analysis.
\end{abstract}

\maketitle
%------------------------------------------------------------------------------------------------------------------------------------------
\section{Introduction}
%------------------------------------------------------------------------------------------------------------------------------------------
	
Fourth-order problems are prominent in the theory of partial differential equations (PDEs), modelling physical phenomena such as the
control of large flexible structures, bridge suspension, microelectromechanical systems, thin-plate elasticity, the Cahn-Hilliard phase-field model, and hyperviscous effects in fluid models.
A prototypical fourth-order problem is the biharmonic problem, which arises in modelling the isotropic behaviour of thin plates.
	
Since the introduction of the globally $\mathcal{C}^1$-conforming Argyris element in the 1960s~\cite{argyris1968tuba},
the biharmonic problem has been widely studied in the context of finite element methods; see also~\cite{dupont1979}.
However, partly due to the sheer technicality of implementing $\mathcal{C}^1$-conforming elements, several mixed and nonconforming approaches have been developed over the years.
These impose lower smoothness requirements on the discrete function spaces, typically at the expense of larger or less well conditioned linear systems.
For instance, families of $\mathcal C^0$-elements for Kirchhoff plates were developed in~\cite{BeiraoKirchhoff-apriori, BrennerSung2005, engel2002continuous}; see also~\cite{BrezziFortin, DestuynderSalaun} and the references therein.
Discontinuous Galerkin (dG) methods have been employed, also in $\h\p$ form, in e.g. \cite{baker1977finite, MozolevskiSuli-CMAME, MozolevskiSuli-CMAM, MozolevskiSuli-JSC, GeorgoulisHoustonIPdGhp, DongIPdGhp,gudi2008mixed}.
	
Computable \emph{a posteriori} error estimates and adaptivity for fourth order problems have received increasing attention over the last twenty years.
For instance, we recall the conforming approximations of problems involving the biharmonic operator of~\cite{neittaanmaki2001posteriori},
the treatment of Morley plates~\cite{BMRMorleyAPos,HuShi-MorleyApos},
quadratic $\mathcal C^0$-conforming interior penalty methods~\cite{biharmonic-Gudi-Brenner-Sung} and general order dG methods~\cite{Virtanen:Apos-biharmonic} for the biharmonic problem,
continuous and dG approximations of the Kirchhoff-Love plate~\cite{hansbo2011posteriori},
the dichotomy principle in a posteriori error estimates for fourth order problems~\cite{adjerid2002posteriori}, and
the Ciarlet-Raviart formulation of the first biharmonic problem~\cite{charbonneau1997residual}.
	
The central difficulty in employing conventional techniques to derive a posteriori error estimates for dG and $\mathcal C^0$-conforming interior penalty methods for the biharmonic problem
lies in constructing an averaging operator to a $\mathcal C^1$-conforming finite element space.
Such an operator, which must satisfy optimal $\h\p$-approximation properties, is required to enable the stability of the continuous PDE operator to be applied to the error.
An $\h\p$-version a posteriori error estimator for biharmonic problems has been presented in~\cite{Schroeder:3field}, relying on the assumption of the existence of the above averaging operator with optimal $\h\p$-approximation properties.
In 2D, the averaging operator may be constructed for arbitrary polynomial degrees, based on conventional macro elements, see e.g.~\cite{GeorgoulisHoustonIPdGhp,BrennerSung2005},
while on tetrahedral meshes this is only possible for $\p=3$.
However, also in 2D, an explicit analysis of optimal $\h\p$-approximation estimates is not available.

An alternative approach, recently proposed in~\cite{kawecki2020unified} in the context of nonlinear PDEs in nondivergence form,
is to reconstruct the solution into $\mathcal{C}^{1}$-conforming spaces introduced in~\cite{brenner2019virtual, neilan2019discrete}.
While this allows us to avoid problems with element geometries, it introduces the disadvantage in the current context that the resulting error estimate would gain an additional suboptimality of order $p^d$ in $d$ spatial dimensions,
due to the repeated application of a polynomial inverse estimate apparently necessary for the analysis.
This approach is discussed further in Remark~\ref{averaging operator} below.

The contribution of this paper is to give an explicit analysis of a residual-based a posteriori error estimator
for a novel $\h\p$-version dG discretisation of the biharmonic problem.
In particular, our analysis does not require a $\mathcal{C}^1$-averaging operator, simultaneously addresses both 2D and 3D, and incorporates arbitrary polynomial degrees which may be variable over simplicial and tensor product meshes.
Instead, the proof of the fact that the estimator forms an upper bound on the error is based on an elliptic reconstruction of the dG solution to $H^2$ and a generalised Helmholtz decomposition of the error, as used in error estimates for classical nonconforming elements~\cite{BMRMorleyAPos,CarstenGallistlHu-Morley}.
We further prove that the estimator forms a local lower bound on the error,
using several $\h\p$-explicit polynomial inverse estimates involving bubble functions and extension operators inspired by those of~\cite{MelenkWohlmuth_hpFEMaposteriori, melenk2003hp}.
The resulting lower bound is algebraically suboptimal with respect to the polynomial degree, guaranteeing that the estimator retains the same exponential convergence properties as the error for problems with point or edge singularities.
	
The analysis focusses on 2D and 3D meshes without hanging nodes, although the case of parallelogram or parallelepiped elements with hanging nodes is addressed in Remark~\ref{handing-nodes}.
However, our analysis does not appear to directly extend to the case of simplicial meshes with hanging nodes due to certain missing technical
results regarding the influence of hanging nodes on $\mathcal C^0$-conforming $\h\p$-version quasi-interpolation operators for $H^2$-functions.
Instead, we numerically demonstrate that the presence of hanging nodes has apparently little effect on the resulting scheme or estimator.

Arguments similar to those presented in this paper may be used to prove upper and lower bounds for the estimator for the $\mathcal C^0$-conforming interior penalty methods in~\cite{biharmonic-Gudi-Brenner-Sung, BrennerSung2005};
see Remark~\ref{remark:basta} below.

\paragraph*{\textbf{Outline of the paper.}}
The formulation of the biharmonic problem and its discretisation via an interior penalty dG scheme is presented in Section~\ref{section:IPdG}.
Section~\ref{section:technical-results} contains certain $\h\p$-explicit approximation results and polynomial inverse and extension results required to derive the error estimate.
The derivation of a computable error estimator that provides a local bound on the error, which is explicit in terms of the polynomial degree, is the topic of Section~\ref{section:apos}.
We present 2D and 3D numerical results in Section~\ref{section:nr} and draw some conclusions in Section~\ref{section:conclusions}.
	
\paragraph*{\textbf{Notation.}}
We adopt standard notation for Sobolev spaces; see e.g.~\cite{adamsfournier}.
Given $D \subset \mathbb R^d$,  $d=2$ or $3$, we denote the Sobolev space of order $s \in \mathbb R$ over $D$ by $H^s(D)$, and let~$(\cdot, \cdot)_{s,D}$, $\Vert \cdot \Vert_{s,D}$, and~$\vert \cdot \vert_{s,D}$,denote its associated inner product, norm and seminorm, respectively.

Let~$\nabla$ denote the gradient operator, and define the Laplacian~$\Delta = \nabla \cdot \nabla$, the bilaplacian~$\Delta^2$, and the Hessian matrix~$\D^2 = \nabla \nabla^{\top}$ operators.
Given $\phi \in H^1(\Omega)$ and $\psibf \in [H^1(\Omega)]^{3}$, the vector-valued curl operator is defined as
\[
\curl({\phi}) =  (-\partial_y \phi , 	 \partial_x  \phi)^{\top}, \quad \quad  \curl(\psibf) =  (\partial_y \psi_3 - \partial _z \psi_2, \partial_z \psi_1 - \partial_x \psi_3, \partial _x\psi_2 - \partial _y \psi_3)^{\top}.
\]
For $\vbf \in [H^1(\Omega)]^2$ with $\vbf = (v_1, v_2)^\top$
and $\wbf \in [H^1(\Omega)]^{3\times 3}$ with rows $\wbf_1, \wbf_2$, and $\wbf_3$, the matrix-valued $\curlbf$ operator is defined as
\[
\curlbf(\vbf) = 
\begin{bmatrix}
\curl (v_1), ~\curl(v_2)
\end{bmatrix}^{\top}, \quad  \quad
\curlbf(\wbf) = 
\begin{bmatrix}
\curl (\wbf_1^{\top}),~ \curl(\wbf_2^{\top}), ~ \curl(\wbf_3^{\top})
\end{bmatrix}^{\top}.
\]
Throughout, $c$ denotes a generic positive constant, which is independent of any discretisation parameters, but may depend on the dimension and shape-regularity constants of the mesh.

%------------------------------------------------------------------------------------------------------------------------------------------
\section{An interior penalty dG method for the biharmonic problem} \label{section:IPdG}
%------------------------------------------------------------------------------------------------------------------------------------------
We present the formulation of the biharmonic problem in Section~\ref{sec:pde}, and introduce a novel interior penalty dG (IPdG) scheme in Section~\ref{subsection:method}.
The scheme is based on a mesh satisfying certain assumptions, which are discussed in Section~\ref{subsection:meshes-degrees}.
	
\subsection{\textbf{The biharmonic problem.}}\label{sec:pde}
Let $\Omega \subset \mathbb R^d$, $d=2$ or $3$, be a bounded polygonal/polyhedral domain and~$\f \in L^2(\Omega)$.
The biharmonic problem reads: find a sufficiently smooth $u : \Omega \rightarrow \mathbb R$ such that
\begin{equation} \label{biharmonic-problem-strong}
\begin{cases}
\Delta^2 u = \f 	& \text{in } \Omega\\
u =\nbfOmega \cdot \nabla u =0				& \text{on } \partial \Omega,
\end{cases}
\end{equation}
where~$\nbfOmega$ denotes the unit outward normal vector on~$\partial \Omega$.
Define
\[
\V := H^2_0(\Omega),\quad \quad \B(u,v) := \int_\Omega \D^2 u : \D^2 v,
\]
where~$:$ denotes tensor contraction. A weak formulation of~\eqref{biharmonic-problem-strong} reads: find $u \in \V$ such that
\begin{equation} \label{biharmonic-problem-weak}
\B(u,v) = (\f,v)_{0,\Omega} \quad \quad \forall v \in \V.
\end{equation}
The well posedness of problem~\eqref{biharmonic-problem-weak} is proven e.g. in~\cite[Section~5.9]{BrennerScott}.
Inhomogeneous boundary data can be addressed following~\cite{beirao2010nonhomo}; see also Remark~\ref{Inhomogous-BC} below.
	
%%%%%%%
\subsection{Meshes and polynomial degree distribution} \label{subsection:meshes-degrees}
%%%%%%%
We consider sequences of decompositions~$\taun$ of~$\Omega$ into disjoint shape-regular triangles or parallelograms in 2D, and tetrahedra or parallelepipeds in 3D.
The set of faces of $\taun$ is denoted by $\En$, and is split into a set of boundary faces $\EnB$, which lie on $\partial \Omega$, and internal faces $\EnI = \En \setminus \EnB$.
The shape-regularity assumption implies that mesh~$\taun$ is locally quasi-uniform, i.e.,
there exists a constant $c_{\operatorname{mesh}} \ge 1$ such that, for elements $\E_1, \E_2 \in \taun$ with $\overline{\E}_1 \cap \overline{\E}_2 \ne \emptyset $,
\begin{equation} \label{quasi-uniformity:assumption}
c_{\operatorname{mesh}}^{-1} \h_{\E_1} \le \h_{\E_2} \le c_{\operatorname{mesh}} \h_{\E_1}.
\end{equation}
Here, $\hE$ and $\hF$ denote the diameter of the element $\E \in \taun$, and the face $\F \in \En$, respectively.
These local mesh sizes form the piecewise constant mesh size function~$\hbf: \Omega \rightarrow \mathbb R^+$ given by
\[
\hbf (\xbf) :=
\begin{cases}
\hE 	& \text{if } \xbf \in \E \text{ for some $\E \in \taun$} \\
\hF	 	& \text{if } \xbf \in \F \text{ for some $\F \in \En$}. 
\end{cases}
\]	
The meshes are assumed to contain no hanging nodes, although a technical argument outlined in Remark~\ref{handing-nodes} below extends our results to cover hanging nodes in parallelogram or parallelepiped meshes.
	
Another consequence of the shape-regularity of~$\taun$ is the existence of a shape-regular kite~$\Etilde$ associated with each internal face $\F \in \EnI$ such that $\Etilde \subset \E_1 \cup \E_2$,
where $\E_1, \E_2 \in \taun$ are the elements meeting at $\F$, and $\Etilde$ is symmetric with respect to $\F$.
An example of such a kite is illustrated in Figure~\ref{figure:kites} for {triangular elements}.
A 3D kite may be constructed analogously, as a hexahedron for tetrahedral meshes or an octahedron for {cubic} meshes.
		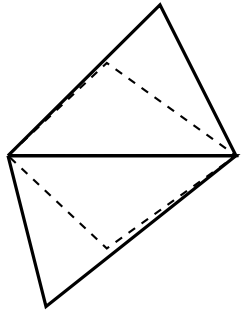
\begin{figure}
			\centering
			\begin{minipage}{0.30\textwidth}
				\begin{center}
					\begin{tikzpicture}[scale=1]
					\draw[black, very thick, -] (0,0) -- (3,0) -- (2,2) -- (0,0);
					\draw[black, very thick, -] (0,0) -- (3,0) -- (0.5,-2) -- (0,0);
					\draw[black, thick, dashed] (0,0) -- (3,0) -- (1.3,1.23) -- (0,0);
					\draw[black, thick, dashed] (0,0) -- (3,0) -- (1.3,-1.23) -- (0,0);
					\end{tikzpicture}
				\end{center}
			\end{minipage}
			\caption{Kites for triangular meshes. The continuous thick lines form the two elements, and the dashed line denotes the perimeter of the kite.}  \label{figure:kites}
		\end{figure}
	
The numerical scheme requires an integer polynomial degree $\pE \geq 2$ associated with each $\E \in \taun$.
We suppose that there exists a constant $c_\p \ge 1$ such that, 
for all $\E_1, \E_2 \in \taun$ with $\overline{\E}_1 \cap \overline{\E}_2 \ne \emptyset $,
	\begin{equation} \label{bounded:polynomial}
		c_{p}^{-1} p_{\E_1} \le p_{\E_2} \le c_{p} p_{\E_1}.
	\end{equation}
The local polynomial degrees are collected by the piecewise constant function~$\pbf: \Omega \rightarrow \mathbb R^+$ with
\[
\pbf (\xbf) :=
\begin{cases}
\pE 				& \text{if } \xbf \in \E \text{ for some } \E \in \taun, \\
\max(p_{\E_1}, p_{\E_2})	& \text{if } \xbf \in \F \in \EnI, \text{ where } \E_1, \E_2 \in \taun \text{ meet at } \F, \\
\pE					& \text{if } \xbf \in \F \in \EnB, \text{ where } \F \text{ is a face of } \E.
\end{cases}
\]
Finally, we define face jump and average operators.
Let $v$ be a scalar-, vector-, or matrix-valued function on $\Omega$, smooth on each $\E \in \taun$ but possibly discontinuous across each $\F \in \EnI$.
For $\F \in \EnI$, let~$\Ep$ and~$\Em$ be the two mesh elements meeting at~$\F$, and let~$v^+$ and~$v^-$ denote the restriction of~$v$ to~$\Ep$ and~$\Em$, respectively.
The face average and jump operators on $\F$ are given by
\[
\{v \}(\xbf) :=  \frac{1}{2}(v^+(\xbf) + v^-(\xbf)), \quad  \quad \llbracket v \rrbracket(\xbf) := v^+(\xbf) - v^-(\xbf) \quad \forall \xbf \in \F,
\]
respectively. 
These definitions are extended to boundary faces where $\{v \}(\xbf) = \llbracket v \rrbracket(\xbf) = v(\xbf)$.
	
\begin{remark} \label{remark:mixed-meshes}
To simplify the notation, we avoid considering mixed meshes of simplicial and tensor product elements.
Although we do not expect such meshes to pose significant difficulties, this assumption allows us to simplify the presentation of the quasi-interpolant in Proposition~\ref{proposition:BabuSuri-KarkulikMelenk} below.
\end{remark}

%%%%%%%
\subsection{The interior penalty dG scheme} \label{subsection:method}
%%%%%%%
Given a mesh~$\taun$ and a polynomial degree distribution~$\pbf$, we introduce the dG space of discontinuous piecewise polynomial functions over~$\taun$ as
\[
\Vn := \{\qpbf \in L^2(\Omega) : \qpbf{}_{|\E} \in \mathbb{P}_{\pE}(\E) \text{ for each } \E \in \taun \}.
\]
For future convenience we introduce the following broken norms: given~$s>0$,
\begin{equation} \label{broken-norms}
\Vert \cdot \Vert_{s,\taun}^2 := \sum_{\E \in \taun} \Vert \cdot \Vert^2_{s,\E}.
\end{equation}
We define the lifting operator $\Lcal :\Vn + \V \rightarrow [\Vn]^{d\times d}$, $d=2$, $3$, as
\begin{equation} \label{lifting-operator}
	\int_\Omega \Lcal (\un) : \mathbf{\vn} := \int_{\En} \{\nbf \cdot (\nabla \cdot \mathbf{\vn}) \} \llbracket \un  \rrbracket   -  \int_{\En} \{ (\mathbf{\vn}) \nbf \} \cdot \llbracket \nabla \un \rrbracket   \quad \quad \forall \mathbf{\vn} \in [\Vn]^{d\times d},
\end{equation}
where~$\nbf$ denotes the unit normal vector to a face, with arbitrary orientation on internal faces and directed outward on boundary faces.
Let $\Dn^2$ denote the elementwise Hessian matrix operator, given by $(\Dn^2 v){}_{|\E} = \D^2 (v{}_{|\E})$ on each~$\E \in \taun$.
We also introduce the piecewise constant dG penalisation parameters $\sigmabold:\En \rightarrow \mathbb R^+$ and $\taubold:\En \rightarrow \mathbb R^+$.
Then, we construct the interior penalty dG bilinear form $\Bn : \Vn \times \Vn \to \mathbb{R}$ as
\[
\begin{split}
	\Bn(\un, \vn) 	
	:=
	&\int_\Omega  \Dn^2 \un : \Dn^2 \vn  +	\int_\Omega \Big(	\Lcal (\un) : \Dn^2 \vn + \Lcal(\vn) : \Dn^2 \un 		\Big) 	\\
	&	+ 	\int_{\En} \Big(	\sigmabold \llbracket \un \rrbracket  \llbracket \vn \rrbracket + \taubold \llbracket \nabla \un \rrbracket \cdot \llbracket \nabla \vn \rrbracket \Big).
\end{split}
\]
Observe that the terms in the dG bilinear form involving the lifting operators are equivalent to
\[
\begin{split}
\int_\Omega \Big(	\Lcal (\un) : \Dn^2 \vn + \Lcal(\vn) : \Dn^2 \un \Big) 
& = \int_{\En} \{\nbf \cdot (\nabla \Delta \vn) \} \llbracket \un  \rrbracket   -  \int_{\En} \{ (\Dn^2\vn)  \nbf \} \cdot \llbracket \nabla \un \rrbracket \\
& \quad   + \int_{\En} \{\nbf \cdot (\nabla \Delta \un ) \} \llbracket \vn  \rrbracket   -  \int_{\En} \{ (\Dn^2 \un) \nbf \} \cdot \llbracket \nabla \vn \rrbracket.
\end{split}
\]
We pose the following interior penalty dG scheme for approximating solutions to the biharmonic problem~\eqref{biharmonic-problem-weak}: find $\un \in \Vn$ such that
\begin{equation} \label{IPdGFEM}
	\Bn(\un, \vn) = (\f, \vn)_{0,\Omega} \quad  \forall  \vn \in \Vn.
\end{equation}
For $v \in \Vn + \V$, we define the dG norm associated with the bilinear form $\Bn(\cdot, \cdot)$ as
\begin{equation} \label{dG:norm}
	\Vert v \Vert_{dG} ^2  := \Vert \Dn^2 v \Vert ^2_{0,\Omega} +\sum_{\F \in \En}\Vert \taubold^{\frac{1}{2}} \llbracket \nabla v \rrbracket \Vert^2_{0,\F}+ \sum_{\F \in \En} \Vert \sigmabold^{\frac{1}{2}} \llbracket v \rrbracket \Vert^2_{0,\F}.
\end{equation}
The stability of $\Bn$ in this norm follows from the properties of $\mathcal{L}$ and a suitable choice of $\sigmabold$ and $\taubold$, as encapsulated in the following result,
which may be proven by arguing as in~\cite[Lemma~5.1]{GeorgoulisHoustonIPdGhp}.
\begin{lem}[Stability of the scheme] \label{lemma:lifting-stab}
	There exists a constant $c_s > 0$ such that $\Lcal$ satisfies
	\[
	\Vert \Lcal (\vn) \Vert ^2_{0,\Omega} 	\leq c_s \Big( \Big\Vert \Big( \frac{\pbf^6}{\hbf^3}  \Big)^{\frac{1}{2}} \llbracket \vn \rrbracket \Big\Vert^2_{0,\En}+ 	\Big\Vert \Big( \frac{\pbf^2}{\hbf}  \Big)^{\frac{1}{2}} \llbracket \nabla \vn \rrbracket \Big\Vert^2_{0,\En} \Big).
	\]
	Consequently, if we pick the dG penalisation parameters as
	\begin{equation} \label{dG:stab-parameters}
		\sigmabold = c_{\sigmabold} \frac{\pbf^6}{\hbf^3}, \quad  \quad \taubold = c_{\taubold} \frac{\pbf^2}{\hbf},
	\end{equation}
	with $c_{\sigmabold}$, $c_{\taubold}  \geq 2c_s+\frac{1}{2}$, then the dG bilinear form satisfies
	\[
	\Bn(v, v)\geq \frac{1}{2} \Vert v \Vert_{dG} ^2   , \quad  \quad \Bn(u, v) \le 2 \Vert u \Vert_{dG} \Vert v \Vert_{dG},
	\]
	for all $u$, $v \in V + \Vn$, and discrete problem~\eqref{IPdGFEM} is well posed.
\end{lem}

\begin{remark}
The dG method~\eqref{IPdGFEM} is based on the Hessian weak formulation of the biharmonic problem, rather than the Laplacian formulation used e.g in~\cite{GeorgoulisHoustonIPdGhp}.
In particular, it may be viewed as an extension of the formulation used for the $\mathcal C^0$-conforming interior penalty method in~\cite{BrennerSung2005}, with additional face and penalisation terms to account for the fully discontinuous trial and test functions.
\end{remark}

%%%%%%%%%%%%%%%%%%%%%%%%%%%%%%%%%%%
\section{$\h\p$-explicit polynomial inverse and extension results} \label{section:technical-results}
%%%%%%%%%%%%%%%%%%%%%%%%%%%%%%%%%%%

In this section, we present a variety of $\h\p$-explicit approximation results and polynomial inverse and extension estimates, which are required for the error estimate of Section~\ref{section:apos}.
Throughout, we suppose that $\p \in \mathbb{N}$, and we use $b_{\E}$ to denote the standard bubble function constructed on the polygon or polyhedron $\E$ as the product of the affine functions vanishing on each face of $\E$.

First, we recall the standard trace inequality from e.g.~\cite[Theorem (1.6.6)]{BrennerScott}.
For this, we introduce the concept of chunkiness parameter of a domain~$\omega \subset \mathbb R^d$. We set
\begin{equation} \label{chunkiness:parameter}
\gamma := \frac{\text{diam}(\omega)}{\rho_{\text{max}}},
\end{equation}
where~$\rho_{\text{max}}$ denotes the maximum over the diameter of all possible balls contained in~$\omega$.

The boundedness of chunkiness parameter~\eqref{chunkiness:parameter} of an element~$\E \in  \taun$ is a consequence of the shape-regularity assumption in Section~\ref{subsection:meshes-degrees}.

\begin{prop}[Trace inequality]
Given a bounded Lipschitz domain~$\omega \subset \mathbb{R}^d$ with diameter~$\h$ and bounded chunkiness parameter~\eqref{chunkiness:parameter}, there exists a constant $c>0$ depending only on $\omega$ such that
\begin{equation} \label{standard:trace}
\Vert v \Vert_{0,\partial \omega}^2 \le c \left( \h^{-1}\Vert v \Vert_{0,\omega}^2 + \Vert v \Vert_{0,\omega}\vert v \vert_{1,\omega}\right) \quad \forall v \in H^1(\omega).
\end{equation}
\end{prop}

The following $\h\p$-explicit inverse estimates are well known.
Estimate~\eqref{inverse:L2trace} was proven in~\cite[Theorem~4]{HestavenWarburton-trace} with explicit constants.
On the other hand, the 2D variant of~\eqref{inverse:H1L2} may be found in~\cite[Theorem~4.76]{SchwabpandhpFEM}, and the 3D case follows analogously.
	
\begin{prop}[$\h\p$-explicit inverse estimates] \label{prop:inverse-trace}
Let $\E$ be a shape-regular triangle, parallelogram, tetrahedron, or parallelepiped, with diameter~$\h$, and let~$\F$ be a face of~$\E$ with diameter scaling as~$\h$.
Then, there exists a constant $c > 0$ independent of~$\h$ or~$\p$ such that,
for all $\qp \in \mathbb P_\p(\E)$ if $\E$ is a triangle or tetrahedron, or $\qp \in \mathbb Q_\p(\E)$ if $\E$ is a parallelogram or parallelepiped,
		\begin{equation} \label{inverse:L2trace}
			\Vert \qp \Vert_{0,\F} \le c p h^{-\frac{1}{2}} \Vert \qp \Vert_{0,\E}.
		\end{equation}
		and
		\begin{equation} \label{inverse:H1L2}
			\vert \qp \vert_{1,\E} \le c \p^2 \h^{-1} \Vert \qp \Vert_{0,\E}.
		\end{equation}
	\end{prop}
	
The analysis in Section~\ref{section:apos} below requires a $\mathcal C^0$-conforming $\h\p$-quasi-interpolant for functions that are not necessarily smooth.
For this, we use a generalisation of Babu\v ska-Suri operator~\cite[Lemma~4.5]{babuskasurihpversionFEMwithquasiuniformmesh},
constructed by combining it with the Karkulik-Melenk smoothing techniques from~\cite[Section $2$]{KarkulikMelenkCAMWA}.

\begin{prop}[Karkulik-Melenk generalisation of the Babu\v ska-Suri $hp$-quasi-interpolant] \label{proposition:BabuSuri-KarkulikMelenk}
		Given a domain $\Omega \subset \mathbb{R}^d$ partitioned into a mesh $\taun$ of triangles, parallelograms, tetrahedra or parallelepipeds,
		there exists an operator $\I : H^s(\Omega) 
		\to \Vn \cap \mathcal C^0(\overline \Omega)$
		such that
		for all $0\le q \le s$,
		\begin{equation} \label{BabuSuri-KarkulikMelenk-prop}
			\Vert v - \I v \Vert_{q,\taun}\le c \left\Vert \frac{\hbf^{\min(\p+1,s)-q}}{\pbf^{s-q}}  v \right\Vert_{s,\taun} \quad \quad \forall v \in H^s(\Omega),
		\end{equation}
		where the constant $c>0$ is independent of~$\h$ and~$\p$, and the broken norms above are defined in~\eqref{broken-norms}.
	\end{prop}
	
Now, we focus on $\h\p$-explicit inverse inequalities involving bubble functions.
Their proof is based on the following two technical lemmata.
The first may be proven by arguing as in~\cite{bernardi2001error}.
\begin{lem}[$\h\p$-explicit polynomial inverse estimates with bubble functions in 1D] \label{lemma:1D:Jacobi:expansion:1bubble}  
Let $\Ihat = (-1,1)$. Given $0 \le \alpha \le \beta$,  there exists a constant $c > 0$ depending on $\alpha$ and $\beta$ but not~$\p$ such that, for all $\qp \in \mathbb P_\p(\Ihat)$,
\begin{equation} \label{1D:Jacobi:expansion}
\int_{\Ihat} (1-x^2)^{\alpha} \qp(x)^2  \le c \p^{2(\beta-\alpha)} \int _{\Ihat} (1-x^2)^{\beta} \qp(x)^2,
\end{equation}
and
\begin{equation} \label{1bubble}
\int_{\widehat I} (1-x)^{\alpha} \qp(x)^2   \le c \p^{2(\beta-\alpha)} \int_{\widehat I} (1-x)^\beta \qp(x)^2 ,	\quad
\int_{\widehat I} (1+x)^{\alpha} \qp(x)^2  \le c \p^{2(\beta-\alpha)} \int_{\widehat I} (1+x)^\beta \qp(x)^2.
\end{equation}
\end{lem}
The  second technical lemma extends the above two inequalities to a quasi-1D result on the trapezoid
\[
D = D(a,b,d) = \{(x,y)\in \mathbb R^2 : y\in [0,d],\; -1+a\,y\le x \le 1+b\,y\},
\]
where $d \in (0,1)$ and $a$, $b\in \mathbb R$ satisfy $-1+a\,d< 1+b\,d$.
Analogous arguments directly extend this result to 3D trapezoidal polyhedra.
%%%%
\begin{lem}[$\h\p$-polynomial inverse estimate with bubbles in quasi-1D trapezoids]\label{lemma:2bubble}
Assume that~$D$ has diameter~$h_D \approx 1$.
To each $y^* \in [0,d]$, associate the segment $I(y^*)=I^*=[a y^* - 1, 1+ b \,y^*]$, and let $F : I^{*} \to [-1, 1]$ be affine with $F(ay^* - 1) = -1$ and $F(1+ b \,y^*) = 1$.
Introduce $\psiF(x) = 1 - (F(x))^2: I^* \rightarrow \mathbb [0,1]$, and define $\Phi \in \mathcal C^0(\overline D)$ such that, for some $s \in \mathbb N$,
\[
			\begin{split}
				& \Phi(\cdot, y^*) \in \mathbb P_{2s}(I^*) ,
				\quad\quad  
				c_1 \psiF^s(x) \le \Phi (x,y^*) \le c_2 \psiF^s(x)  \quad \forall  x\in I^*,
			\end{split}
\]
where the constants $c_1$, $c_2 > 0$ depend only on $a$, $b$, $\Phi$, and $y^*$.
		
Then, there exists a constant $c > 0$, depending only on $\alpha$, $\beta$, and $\Phi$, such that, for all $\beta>\alpha\ge 0$,
\[
			\Vert \Phi^{\frac{\alpha}{2}} \qp \Vert_{0,D} \le c \, \p^{s(\beta-\alpha)} \Vert \Phi^{\frac{\beta}{2}} \qp \Vert_{0,D} \quad  \forall \qp \in \mathbb P_\p (D).
\]
	\end{lem}
	\begin{proof}
The proof is based on~\cite[Lemma~D.2]{melenk2003hp-preprint}.
By assumption, $\Phi$ is a continuous function of $y^*$. Therefore, $c_1$ and~$c_2$ depend continuously on~$y^*$.
Since $y^* \in [0,d]$, $c_1$ and~$c_2$ attain their extremal values $\overline{c}_1=\min_{y^*\in [0,d]}(c_1(y^*))$ and $\overline{c}_2=\max_{y^*\in [0,d]}(c_2(y^*))$, which satisfy $0 < c_1 \leq c_2$, and
		\begin{equation} \label{punctual:equivalence:without:dependence}
			\overline{c}_1 \psiF^s (x) \le \Phi(x,y^*) \le \overline{c}_2 \psiF^s (x) \quad \forall x \in I^*.
		\end{equation}
Bounds~\eqref{1D:Jacobi:expansion} and~\eqref{punctual:equivalence:without:dependence} imply that
\[
\begin{split}
\int_{I^*} 	\Phi(x,y^*)^{\alpha}  \qp(x,y^*)^2  
&\le \overline{c}_2^{\alpha} \int_{I^*} \psiF^{s \alpha} (x)  \qp(x,y^*)^2 \le c \overline{c}_2^{\alpha} \p^{s(\beta-\alpha)} \int_{I^*} \psiF^{s \beta} (x)  \qp(x,y^*)^2  	\\
& \le 	c \frac{\overline{c}_2^{\alpha}}{\overline c_1^\beta}  (\p+1)^{s(\beta-\alpha)}\int_{I^*}\Phi(x,y^*)^{\beta}\qp(x,y^*)^2 ,
\end{split}
\]
where $c$ is independent of $y^*$.
		The assertion follows by integrating over $y^*\in [0,d]$.
	\end{proof}	

These two technical lemmata enable us to prove an inverse estimate for bubble functions on 2D and 3D elements, using a partitioning argument introduced by Melenk and Wohlmuth~\cite{MelenkWohlmuth_hpFEMaposteriori}; see also~\cite{melenk2003hp-preprint,melenk2003hp}.
\begin{prop}[$hp$-polynomial inverse estimate with bubbles] \label{proposition:inverse-bubble} 
Let $\E \subset \mathbb R^d$, $d=2$ or~$3$, be a triangle, parallelogram, tetrahedron, or parallelepiped.
Then, there exists a constant $c > 0$, independent of~$\h$ and~$\p$, such that, for all $\qp \in \mathbb P_\p(\E)$ if $\E$ is a triangle or tetrahedron, or $\qp \in \mathbb Q_\p(\E)$ if  $\E$ is a parallelogram or parallelepiped,
\begin{equation} \label{inverse:bubble}
\Vert \bE^{\frac{\alpha}{2}} \qp \Vert_{0,\E} \le c \p^{d (\beta - \alpha)} \Vert \bE^{\frac{\beta}{2}} \qp \Vert_{0,\E}, \quad\quad  -\frac{1}{2}< \alpha \leq \beta.
\end{equation}
\end{prop}
\begin{proof}
The proof is similar to that of~\cite[Theorem D2]{melenk2003hp-preprint} and for this reason we only sketch it.
If~$\E$ is a parallelogram or parallelepiped, the result follows from Lemma~\ref{lemma:2bubble}.
		
Suppose~$\E$ is the triangle with vertices $\{(0,0), (1,0), (0,1) \}$. Split~$\E$ into the overlapping subsets
\[	
			\E = \left( \cup_{i=1}^6 D_i   \right) \cup \left(\cup_{i=1}^3 P_i \right) \cup R ,
\]
where~$D_i$, $i=1,\dots, 6$ are the trapezoids depicted in Figure~\ref{figure:trapezoidals}, and $P_i$, $i=1,2,3$, are the parallelograms shown in Figure~\ref{figure:parallelograms}.
The remainder $R$, illustrated in Figure~\ref{figure:remainder}, is separated from $\partial \E$.
The assertion follows by applying Lemma~\ref{lemma:2bubble} on each trapezoid $D_i$, using~\eqref{1bubble} and a tensor product argument on each parallelogram,
and observing that $\bE \approx 1$ on the remainder $R$.
		
When $\E$ is a tetrahedron, we construct a similar overlapping decomposition consisting of parallelepipeds, trapezoidal polyhedra, and a remainder, and apply analogous arguments on each.
The power~$d (\beta-\alpha)$ in~\eqref{inverse:bubble} is due to the application of the tensor product version of inverse estimate~\eqref{1bubble} when dealing with the parallelepipeds.
\begin{figure}
			\centering
			\begin{minipage}{0.30\textwidth}
				\begin{center}
					\begin{tikzpicture}[scale=3.2]
					\draw[black, very thick, -] (0,0) -- (1,0) -- (0,1) -- (0,0);
					\draw[black, -] (5/12,0) -- (17/24, 7/24); \draw[black, -] (7/12,0) -- (19/24, 5/24);
					\draw[black, -] (0,5/12) -- (7/24, 17/24); \draw[black, -] (0, 7/12) -- (5/24, 19/24);
					\fill[orange,opacity=0.2] (5/12,0) -- (17/24, 7/24)  -- (19/24, 5/24) -- (7/12,0) -- (5/12,0);
					\fill[blue,opacity=0.2] (0,5/12) -- (7/24, 17/24)  -- (5/24, 19/24) -- (0, 7/12) -- (0, 5/12);
					\fill[black] (0,5/12) circle(0.025cm); \fill[black] (0,7/12) circle(0.025cm); \fill[black] (5/24,19/24) circle(0.025cm); \fill[black] (7/24,17/24) circle(0.025cm);
					\fill[black] (5/12, 0) circle(0.025cm); \fill[black] (7/12, 0) circle(0.025cm); \fill[black] (17/24, 7/24) circle(0.025cm); \fill[black] (19/24, 5/24) circle(0.025cm);
					\draw (0,5/12) node[black, left] {\tiny{$\left(0,\frac{5}{12}\right)$}}; \draw (0,7/12) node[black, left] {\tiny{$\left(0,\frac{7}{12}\right)$}};
					\draw (5/12,0) node[black, below] {\tiny{$\left(\frac{5}{12}, 0\right)$}}; \draw (7/12,0) node[black, below right] {\tiny{$\left(\frac{7}{12},0\right)$}};
					\draw (19/24, 5/24) node[black, right] {\tiny{$\left(\frac{19}{24},\frac{5}{24}\right)$}}; \draw (17/24,7/24) node[black, above right] {\tiny{$\left(\frac{17}{24},\frac{7}{24}\right)$}};
					\draw (5/24, 19/24) node[black, above right] {\tiny{$\left(\frac{5}{24}, \frac{19}{24}\right)$}}; \draw (7/24,17/24) node[black, right] {\tiny{$\left(\frac{7}{24}, \frac{17}{24}\right)$}};
					\draw(2/9-1/24, 1/2) node[black] {{\footnotesize{$D_1$}}}; \draw(4/8-1/12,1/6-1/24) node[black] {{\footnotesize{$D_2$}}};
					\end{tikzpicture}
				\end{center}
			\end{minipage}
			\begin{minipage}{0.30\textwidth}
				\begin{center}
					\begin{tikzpicture}[scale=3.2]
					\draw[black, very thick, -] (0,0) -- (1,0) -- (0,1) -- (0,0);
					\draw[black, -] (5/12,0) -- (0, 5/24); \draw[black, -] (7/12,0) -- (0, 7/24);
					\draw[black, -] (5/12, 7/12) -- (0, 19/24); \draw[black, -] (7/12, 5/12) -- (0, 17/24);
					\fill[orange,opacity=0.2] (5/12,0) -- (0,5/24)  -- (0,7/24) -- (7/12,0) -- (5/12,0);
					\fill[blue,opacity=0.2] (5/12,7/12) -- (0,19/24)  -- (0,17/24) -- (7/12, 5/12) -- (5/12, 7/12);
					\fill[black] (7/12,5/12) circle(0.025cm); \fill[black] (5/12, 7/12) circle(0.025cm); \fill[black] (0, 5/24) circle(0.025cm); \fill[black] (0, 7/24) circle(0.025cm);
					\fill[black] (5/12, 0) circle(0.025cm); \fill[black] (7/12, 0) circle(0.025cm); \fill[black] (0, 17/24) circle(0.025cm); \fill[black] (0, 19/24) circle(0.025cm);
					\draw (7/12,5/12) node[black, right] {\tiny{$\left(\frac{7}{12},\frac{5}{12}\right)$}}; \draw (5/12,7/12) node[black, right] {\tiny{$\left(\frac{5}{12},\frac{7}{12}\right)$}};
					\draw (5/12,0) node[black, below] {\tiny{$\left(\frac{5}{12}, 0\right)$}}; \draw (7/12,0) node[black, below right] {\tiny{$\left(\frac{7}{12},0\right)$}};
					\draw ((0, 5/24) node[black, below left] {\tiny{$\left(0,\frac{5}{24}\right)$}}; \draw (0,7/24) node[black, left] {\tiny{$\left(0,\frac{7}{24}\right)$}};
					\draw (0, 17/24) node[black, below left] {\tiny{$\left(0, \frac{17}{24}\right)$}}; \draw (0,19/24) node[black, left] {\tiny{$\left(0, \frac{19}{24}\right)$}};
					\draw(1/4,4/8) node[black] {{\footnotesize{$D_3$}}}; \draw(5/16-1/48,1/4-1/24) node[black] {{\footnotesize{$D_4$}}};
					\end{tikzpicture}\end{center}
			\end{minipage}
			\begin{minipage}{0.30\textwidth}
				\begin{center}
					\begin{tikzpicture}[scale=3.2]
					\draw[black, very thick, -] (0,0) -- (1,0) -- (0,1) -- (0,0);
					\draw[black, -] (17/24,0) -- (5/12, 7/12); \draw[black, -] (7/12,5/12) -- (19/24, 0);
					\draw[black, -] (5/24, 0) -- (0, 5/12); \draw[black, -] (0, 7/12) -- (7/24, 0);
					\fill[orange,opacity=0.2] (17/24,0) -- (5/12, 7/12) -- (7/12,5/12) -- (19/24, 0) -- (17/24, 0);
					\fill[blue,opacity=0.2] (5/24, 0) -- (0, 5/12) -- (0, 7/12) -- (7/24, 0) -- (5/24, 0);
					\fill[black] (17/24,0) circle(0.025cm); \fill[black] (5/12, 7/12) circle(0.025cm); \fill[black] (7/12, 5/12) circle(0.025cm); \fill[black] (19/24, 0) circle(0.025cm);
					\fill[black] (5/24, 0) circle(0.025cm); \fill[black] (0, 5/12) circle(0.025cm); \fill[black] (0, 7/12) circle(0.025cm); \fill[black] (7/24, 0) circle(0.025cm);
					\draw (7/12,5/12) node[black, right] {\tiny{$\left(\frac{7}{12},\frac{5}{12}\right)$}}; \draw (5/12,7/12) node[black, right] {\tiny{$\left(\frac{5}{12},\frac{7}{12}\right)$}};
					\draw (5/24,0) node[black, below left] {\tiny{$\left(\frac{5}{24}, 0\right)$}}; \draw (7/24,0) node[black, below] {\tiny{$\left(\frac{7}{24},0\right)$}};
					\draw ((19/24, 0) node[black, below right] {\tiny{$\left(\frac{5}{24}, 0\right)$}}; \draw (0,7/24) node[black, above left] {\tiny{$\left(0,\frac{7}{24}\right)$}};
					\draw (0, 17/24) node[black, below left] {\tiny{$\left(0, \frac{17}{24}\right)$}}; \draw (17/24,0) node[black, below] {\tiny{$\left( \frac{17}{24}, 0\right)$}};
					\draw(5/16,1/6-1/24) node[black] {{\footnotesize{$D_5$}}}; \draw(7/16,3/8) node[black] {{\footnotesize{$D_6$}}};
					\end{tikzpicture}
				\end{center}
			\end{minipage}
			\caption{Trapezoids $D_i$, $i=1,\dots,6$.} \label{figure:trapezoidals}
		\end{figure}
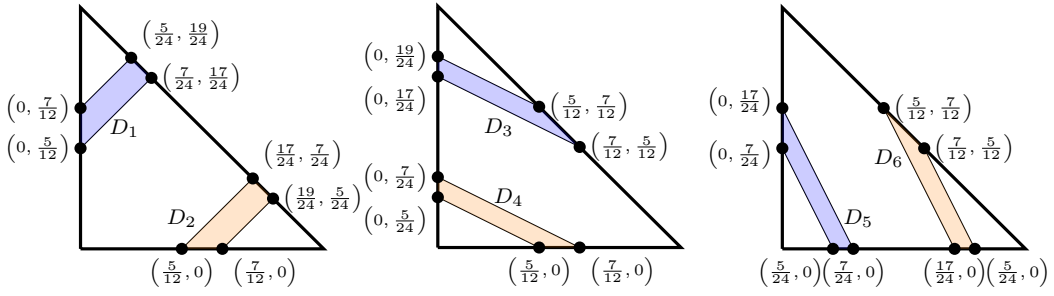
%%%
		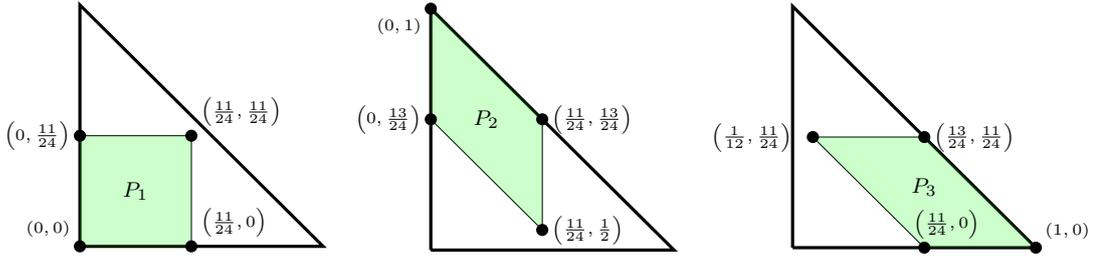
\begin{figure}
			\centering
			\begin{minipage}{0.30\textwidth}
				\begin{center}
					\begin{tikzpicture}[scale=3.2]
					\draw[black, very thick, -] (0,0) -- (1,0) -- (0,1) -- (0,0);
					\draw[black, -] (11/24,0) -- (11/24, 11/24) -- (0,11/24); 
					\fill[green,opacity=0.2] (11/24,0) -- (11/24, 11/24) -- (0,11/24) -- (0,0) -- (11/24,0);
					\fill[black] (11/24,0) circle(0.025cm); \fill[black] (11/24, 11/24) circle(0.025cm); \fill[black] (0,11/24) circle(0.025cm); \fill[black] (0, 0) circle(0.025cm);
					\draw (11/24,0) node[black, above right] {\tiny{$\left(\frac{11}{24}, 0\right)$}}; \draw (11/24,11/24) node[black, above right] {\tiny{$\left(\frac{11}{24},\frac{11}{24}\right)$}};
					\draw (0,11/24) node[black, left] {\tiny{$\left( 0,\frac{11}{24}\right)$}}; \draw (0,0) node[black, above left] {\tiny{$(0, 0)$}};
					\draw(11/48,11/48) node[black] {{\footnotesize{$P_1$}}};
					\end{tikzpicture}
				\end{center}
			\end{minipage}
			\begin{minipage}{0.30\textwidth}
				\begin{center}
					\begin{tikzpicture}[scale=3.2]
					\draw[black, very thick, -] (0,0) -- (1,0) -- (0,1) -- (0,0);
					\draw[black, -] (0, 13/24) -- (11/24, 1/12) -- (11/24, 13/24);
					\fill[green,opacity=0.2] (0, 13/24) -- (11/24, 1/12) -- (11/24, 13/24) -- (0,1) -- (0,13/24);
					\fill[black] (0, 13/24) circle(0.025cm); \fill[black] (11/24, 1/12) circle(0.025cm); \fill[black] (11/24,13/24) circle(0.025cm); \fill[black] (0, 1) circle(0.025cm);
					\draw (11/24,1/12) node[black, right] {\tiny{$\left(\frac{11}{24}, \frac{1}{2}\right)$}}; \draw (11/24,13/24) node[black, right] {\tiny{$\left(\frac{11}{24},\frac{13}{24}\right)$}};
					\draw (0, 1) node[black, below left] {\tiny{$\left( 0, 1\right)$}}; \draw (0, 13/24) node[black, left] {\tiny{$\left( 0, \frac{13}{24} \right)$}};
					\draw(11/48,13/24) node[black] {{\footnotesize{$P_2$}}};
					\end{tikzpicture}\end{center}
			\end{minipage}
			\begin{minipage}{0.30\textwidth}
				\begin{center}
					\begin{tikzpicture}[scale=3.2]
					\draw[black, very thick, -] (0,0) -- (1,0) -- (0,1) -- (0,0);
					\draw[black, -] (13/24,0) -- (1/12, 11/24) -- (13/24, 11/24);
					\fill[green,opacity=0.2] (13/24,0) -- (1/12, 11/24) -- (13/24, 11/24) -- (1,0) -- (13/24,0);
					\fill[black] (13/24, 0) circle(0.025cm); \fill[black] (1/12, 11/24) circle(0.025cm); \fill[black] (13/24,11/24) circle(0.025cm); \fill[black] (1, 0) circle(0.025cm);
					\draw (11/24, 0) node[black, above right] {\tiny{$\left(\frac{11}{24}, 0 \right)$}}; \draw (13/24,11/24) node[black, right] {\tiny{$\left(\frac{13}{24},\frac{11}{24}\right)$}};
					\draw (1,0) node[black, above right] {\tiny{$\left( 1, 0\right)$}}; \draw (1/12, 11/24) node[black, left] {\tiny{$\;\;\left( \frac{1}{12}, \frac{11}{24} \right)\;\;$}};
					\draw(7/12-1/24, 1/4) node[black] {{\footnotesize{$P_3$}}};
					\end{tikzpicture}
				\end{center}
			\end{minipage}
			\caption{Parallelograms $P_i$, $i=1,2,3$. } \label{figure:parallelograms}
		\end{figure}
		
		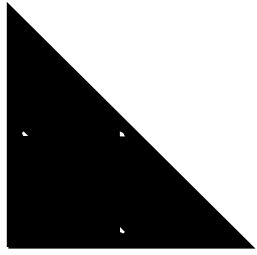
\begin{figure}
			\centering
			\begin{minipage}{0.30\textwidth}
				\begin{center}
					\begin{tikzpicture}[scale=3.2]
					\draw[black, very thick, -] (0,0) -- (1,0) -- (0,1) -- (0,0);
					\fill[black] (5/12,0) -- (17/24, 7/24)  -- (19/24, 5/24) -- (7/12,0) -- (5/12,0);
					\fill[black] (0,5/12) -- (7/24, 17/24)  -- (5/24, 19/24) -- (0, 7/12) -- (0, 5/12);
					\fill[black] (5/12,0) -- (0,5/24)  -- (0,7/24) -- (7/12,0) -- (5/12,0);
					\fill[black] (5/12,7/12) -- (0,19/24)  -- (0,17/24) -- (7/12, 5/12) -- (5/12, 7/12);
					\fill[black] (17/24,0) -- (5/12, 7/12) -- (7/12,5/12) -- (19/24, 0) -- (17/24, 0);
					\fill[black] (5/24, 0) -- (0, 5/12) -- (0, 7/12) -- (7/24, 0) -- (5/24, 0);
					\fill[black] (11/24,0) -- (11/24, 11/24) -- (0,11/24) -- (0,0) -- (11/24,0);
					\fill[black] (0, 13/24) -- (11/24, 1/12) -- (11/24, 13/24) -- (0,1) -- (0,13/24);
					\fill[black] (13/24,0) -- (1/12, 11/24) -- (13/24, 11/24) -- (1,0) -- (13/24,0);
					\end{tikzpicture}
				\end{center}
			\end{minipage}
			\caption{The three white ``small holes'' inside triangle $\E$ denote the remainder $R$.} \label{figure:remainder}
		\end{figure}
	\end{proof}

In Proposition~\ref{proposition:inverse-bubble}, we deduced the same $\h\p$-polynomial inverse estimate as in~\cite[Proposition~3.45]{verfurth2013posteriori-book}.
However, our proof extends to polygonal and polyhedral elements as well.

	The following $\h\p$-explicit polynomial weighted inverse estimate is proven in~\cite[Propositions~3.85, 3.86]{verfurth2013posteriori-book} with explicit constants in both 2D and 3D.
	\begin{prop}[$hp$-explicit polynomial weighted~$H^1$ to~$L^2$ inverse estimate] \label{proposition:weighted-inverse-H1L2}
Let~$\E$ be a triangle, parallelogram, a tetrahedron or parallelepiped with diameter~$\h$.
Then, there exists a constant $c > 0$, independent of~$\h$ and~$\p$, such that
		for all $\qp \in \mathbb P_\p(\E)$ if $\E$ is a triangle or tetrahedron, and for all $\qp \in \mathbb Q_\p(\E)$ if $\E$ is a parallelogram or parallelepiped,
		\begin{equation} \label{inverse:weighted-H1L2}
			\Vert \nabla (\bE^{} \qp) \Vert_{0,\E} \le c \frac{\p}{\h} \Vert \qp \Vert_{0,\E}.
		\end{equation}
	\end{prop}

Next, we prove an $\h\p$-explicit polynomial extension stability result.
\begin{prop}[$\varepsilon$-weighted $\h\p$-explicit polynomial extension stability result] \label{proposition:extension-inverse}
Let $\E$ be a shape-regular triangle, parallelogram, tetrahedron, or parallelepiped with diameter $\hE$, and let $\F$ be a face of $\E$.
The shape-regularity implies that~$\hF$ scales like~$\hE$.
Associated with face~$\F$, define
		\[
		\Phi_{\F} = \begin{cases}
		\bF & \text{if } \E \text{ is a triangle/tetrahedron}\\
		1 & \text{if } \E \text{ is a parallelogram/parallelepiped},
		\end{cases}
		\]
where~$\bF$ denotes the standard bubble function associated with face~$\F$.
Then, in 2D and in 3D when~$\F$ is a triangle,
there exists an extension operator $\Eit : \mathbb P_{\p}(\F) \rightarrow H^2(\E)$ such that, for all $\qp \in \mathbb P_\p(\F)$ and for all $\varepsilon>0$ sufficiently small, there exists a constant $c > 0$, independent of~$\h$ and~$\p$, such that
\begin{align}
& \Eit(\qp){}_{|\F} = \Phi_{\F} \qp{}_{|\F} \quad \text{on }\F,  \label{extension-1} \\
& \Vert \Eit(\qp) \Vert_{0,\E} \le c\hE^{\frac{1}{2}}\varepsilon^{\frac{1}{2}} \Vert \qp \Vert_{0,\F} ,\label{extension-2}\\
& \vert \Eit(\qp) \vert_{1,\E} \le c \hE^{-\frac{1}{2}}(\varepsilon \p^4 + \varepsilon^{-1})^{\frac{1}{2}} \Vert \qp \Vert_{0,\F}, \label{extension-3} \\
& \Vert D^2 \Eit(\qp) \Vert_{0,\E} \le c \hE^{-\frac{3}{2}} (\varepsilon \p^8 + \varepsilon^{-3} + \varepsilon^{-1} \p^4) ^{\frac{1}{2}} \Vert \qp \Vert_{0,\F}.   \label{extension-4}
		\end{align}
In 3D when~$\F$ is a parallelogram, the above bounds are valid substituting~$\mathbb P_\p(\F)$ with~$\mathbb Q_\p(\F)$.
\end{prop}
	\begin{proof}
We suppose that $\E = \Ehat$ is the reference element with diameter $\h_{\Ehat}=1$, and that $\F=\Fhat$ has size $1$. 
		The general case then follows by a scaling argument.
		
		When $\Ehat$ is a parallelogram, the proof is based on that of~\cite[Lemma~2.6]{MelenkWohlmuth_hpFEMaposteriori}.
		For the sake of completeness, we sketch the proof when $\Ehat = [0,1]^2$ and $\Fhat = [0,1] \times \{0\}$.
		
Introducing the extension operator
\[
\Eit(\qp) = \qp (1-y) e^{-\frac{y}{\varepsilon}},
\]
the properties~\eqref{extension-1} and~\eqref{extension-2} are immediate. To show~\eqref{extension-3}, Lemma~\ref{proposition:weighted-inverse-H1L2} implies that
\[
\Vert \partial _x \Eit(\qp) \Vert^2_{0,\Ehat} = \Vert \partial _x \qp \Vert^2_{0,\Fhat} \Vert (1-y) e^{-\frac{y}{\varepsilon}}\Vert ^2_{0,[0,1]} \le c \p^4 \varepsilon \Vert \qp \Vert^2_{0,\Fhat},
\]
and moreover
\[
		\Vert \partial _y \Eit(\qp) \Vert^2_{0,\Ehat} 	\le c \Vert \qp \Vert^2_{0,\Fhat} \left( \Vert e^{-\frac{y}{\varepsilon}} \Vert^2_{0,[0,1]} + \frac{1}{\varepsilon^2} \Vert e^{-\frac{y}{\varepsilon}} \Vert^2_{0,[0,1]} \right)
		\le c \varepsilon ^{-1} \Vert \qp \Vert^2_{0,\Fhat} .
\]
We prove~\eqref{extension-4} analogously, observing that Lemma~\ref{proposition:weighted-inverse-H1L2} provides
\[
\Vert \partial _{xx} \Eit(\qp) \Vert^2_{0,\Ehat} \le c \p^8 \varepsilon \Vert \qp \Vert^2_{0,\Fhat},
\]
and we directly obtain
\[
		\Vert \partial _{yy} \Eit(\qp) \Vert^2_{0,\Ehat} 	=   \Vert \qp \Vert^2_{0,\Fhat} \Vert \partial _{yy} ((1-y) e^{-\frac{y}{\varepsilon}})\Vert ^2_{0,[0,1]}
		\le c \varepsilon^{-3} \Vert \qp \Vert^2_{0,\Fhat},
\]
and
\[
		\Vert \partial _{xy} \Eit(\qp) \Vert^2_{0,\Ehat}  	= \Vert \partial _x \qp \Vert^2_{0,\Fhat}  \Vert \partial _{y} ((1-y) e^{-\frac{y}{\varepsilon}})\Vert ^2_{0,[0,1]}
		\le c \p^4 \varepsilon \Vert \qp \Vert^2_{0,\Fhat}.
\]
		Next, suppose $\Ehat$ is the triangle with vertices $\{(0,0), (1,0), (0,1)\}$, and,
		without loss of generality, take $\Fhat = [0,1] \times \{0\}$.
		Defining the extension operator
\[
			\Eit(\qp) := \qp x (1-x-y) e^{-\frac{y}{\varepsilon}},
\]
property~\eqref{extension-1} is once again immediate and~\eqref{extension-2} is valid because
\[
\Vert \Eit (\qp) \Vert^2_{0,\Ehat} 	 = \int_0 ^1 \int_0 ^{1-x} \qp(x)^2 x^2 (1-x-y)^2 e^{-2\frac{y}{\varepsilon}}\dy\, \dx \le \int_0^1 \qp^2(x)\dx \int_{0}^1 e^{-2\frac{y}{\varepsilon}} \dy \le c \varepsilon \Vert \qp \Vert^2_{0,\Fhat}.
\]
To show~\eqref{extension-3}, denote the derivative of~$\qp$ with respect to the local coordinate system on $\Fhat$ by~$\qp'$.
Expanding the integral as before and applying Lemma~\ref{prop:inverse-trace} implies that
\[
\begin{split}
\Vert \partial _x \Eit(\qp) \Vert^2_{0,\Ehat} 	& \le \int_0^1 \qp'(x)^2  	\int_0^{1-x} e^{-2\frac{y}{\varepsilon}}  \dy \, \dx + 	2\int_0^1 \qp(x)^2 \int_0^{1-x} e^{-2\frac{y}{\varepsilon}} \dy  \, \dx\\
														& \le c \varepsilon (\Vert \qp'\Vert^2_{0,\Fhat} + \Vert \qp \Vert^2_{0,\Fhat}) \le  c \varepsilon \p^4 \Vert \qp \Vert^2_{0,\Fhat},
\end{split}
\]
and similarly
\[
\begin{split}
\Vert \partial _y \Eit(\qp) \Vert^2_{0,\Ehat} 	& \le 	\int_0^1 (1 + \varepsilon^{-2})	\qp(x)^2  \int_0^{1-x} e^{-2\frac{y}{\varepsilon}}  \dy \,\dx \le c (\varepsilon + \varepsilon^{-1}) \Vert \qp\Vert^2_{0,\Fhat} \le  c \varepsilon^{-1} \Vert \qp \Vert^2_{0,\Fhat}.
\end{split}
\]
Finally, to prove~\eqref{extension-4} we note that
\[
\begin{split}
\Vert \partial_{xx} \Eit (\qp) \Vert^2_{0,\Ehat}	& \le c \int_0^1 (\qp''(x) ^2 + \qp'(x) ^2 +\qp(x) ^2) \int_{0}^{1-x} e^{-2\frac{y}{\varepsilon}}\dy\,\dx \\
															& \le c \varepsilon  \left( \p^8 + \p^4 + 1 \right)  \Vert \qp \Vert^2_{0,\Fhat}  \le c \varepsilon \p^8 \Vert \qp \Vert^2_{0,\Fhat},
\end{split}
\]
and
\[
\begin{split}
\Vert \partial_{yy} \Eit (\qp) \Vert^2_{0,\Ehat}	& \le c \int_0^1 \qp(x) ^2 \int_{0}^{1-x} ( \varepsilon^{-2}(1-x-y)e^{-\frac{y}{\varepsilon}})^2  \dy  \dx \le c \varepsilon ^{-3} \Vert \qp \Vert^2_{0,\Fhat}.
\end{split}
\]
Similarly,
\[
\begin{split}
\Vert \partial_{xy} \Eit (\qp) \Vert^2_{0,\Ehat} 	& \le c \int_0^1 \int_{0}^{1-x}\big( (1 + \varepsilon^{-1} (1-x-y)) (q(x) + x q'(x)) + \varepsilon^{-1} x q(x) \big)^2e^{-2 \frac{y}{\varepsilon}}\dy\dx\\
															& \le c \varepsilon ^{-1}  \big( \p^4 \Vert \qp \Vert^2_{0,\Fhat} + \Vert \qp \Vert^2_{0,\Fhat}    \big).
\end{split}
\]
and the assertion follows. 
The 3D case follows by extending the arguments above.
\end{proof}
	
Selecting $\varepsilon=\p^{-2}$, the following result is an immediate consequence of Proposition~\ref{proposition:extension-inverse}.
\begin{cor}[$\h\p$-polynomial extension stability result] \label{corollary:extension-inverse}
Using the same notation and under the same assumptions as in Proposition~\ref{proposition:extension-inverse},
there exists an extension operator $\Eit : \mathbb P_{\p}(\F) \rightarrow H^2(\E)$ such that there exists a constant $c > 0$, independent of~$\h$ and~$\p$, such that, for all $\qp \in \mathbb P_\p(\F)$,
\begin{align} 
& \Eit(\qp) {}_{|\F} = \Phi_{\F} \qp{}_{|\F} \quad \text{on }\F,	\\
& \Vert \Eit(\qp) \Vert_{0,\E} + \hE^{1} \p^{-2} \vert \Eit(\qp) \vert_{1,\E} +\hE^{2} \p^{-4} 	\Vert D^2 \Eit(\qp) \Vert_{0,\E} \leq	c \hE^{\frac{1}{2}} \p^{-1} \Vert \qp \Vert_{0,\F}. \label{eq:extensionoperator}
\end{align}
In 3D when~$\F$ is a parallelogram, the space~$\mathbb P_\p(\F)$ is replaced by $\mathbb Q_\p(\F)$.
\end{cor}
	
%%%%%%%%%%%%%%%%%%%%%%%%%%%%%%%%%%%%%%%%%%%%%%%%%%%%%%%%%%%%%%%%%%%%%%%%%%%
\section{Error estimator and a posteriori error analysis} \label{section:apos}
%%%%%%%%%%%%%%%%%%%%%%%%%%%%%%%%%%%%%%%%%%%%%%%%%%%%%%%%%%%%%%%%%%%%%%%%%%%
In this section, we introduce a computable error estimator, which provides an upper bound and a local lower bound on the error measured in the dG norm in~\eqref{dG:norm}.
These are the results of Theorems~\ref{theorem:reliability} and~\ref{theorem:efficiency} below, respectively.
\begin{definition}[Error estimator]\label{def:estimator}
We introduce the error estimator
\[
\eta^2 := \sum_{\E \in \taun} \etaE^2 \quad\text{ with }\quad \etaE^2 := \eta_{\E,1}^2+ \eta_{\E,2}^2+ \eta_{\E,3}^2+ \eta_{\E,4}^2+ \eta_{\E,5}^2+ \eta_{\E,6}^2,
\]
where
\begin{align*} 
&\eta_{\E,1}^2 := \Big \Vert \Big( \frac{\hbf}{\pbf} \Big)^2 (\f - \Delta^2 \un) \Big \Vert^2_{0,\E},	\quad 
&&	\eta_{\E,2}^2 := \frac{1}{2} \sum_{\F \in \EE \cap \EnI}  	\Big \Vert \Big( \frac{\hbf}{\pbf} \Big)^{\frac{3}{2}} \llbracket \nbf \cdot \nabla \Delta \un \rrbracket \Big\Vert^2_{0, \F},\\
& \eta_{\E,3}^2 := \frac{1}{2} \sum_{\F \in \EE \cap \EnI}  	\Big \Vert \Big( \frac{\hbf}{\pbf} \Big)^{\frac{1}{2}} \llbracket (\D^2 \un) \nbf \rrbracket 	\Big\Vert^2_{0,\F},\quad
&& \eta_{\E,4}^2	:= \frac{1}{2} \sum_{\F \in \EE}  \alphabold_{\F} \Big \Vert \Big( \frac{\hbf}{\pbf} \Big)^{\frac{1}{2}} \llbracket (\D^2 \un) \tbf \rrbracket 	\Big\Vert^2_{0,\F}, \\
& \eta_{\E,5}^2 := \frac{1}{2} \sum_{\F \in \EE}  	\alphabold_{\F} \Vert 	\pbf^{\frac{1}{2}} \taubold^{\frac{1}{2}} \llbracket \nabla \un \rrbracket \Vert^2_{0,\F},\quad 
&&	\eta_{\E,6}^2 := \frac{1}{2} \sum_{\F \in \EE} 	\alphabold_{\F} \Vert \sigmabold^{\frac{1}{2}} \llbracket \un \rrbracket \Vert^2_{0,\F},
\end{align*}
with $\alphabold_{\F} = 2$ for $\F \in \EnB$ and $\alphabold_{\F} = 1$ otherwise.
\end{definition}
The notation for the tangential Hessian $(\D^2 \un)\tbf$ and tangential gradient~$(\nabla \un)\cdot \tbf$ must be defined separately in 2D and 3D.
In 2D, the unit tangential vector on a given face uniquely (up to its sign) satisfies $\tbf \cdot \nbf = 0$.
Hence, for~$\vbf \in \mathbb{R}^2$ and $M \in \mathbb{R}^{2 \times 2}$, the terms $\vbf \cdot \tbf$ and $M \tbf$ have their usual linear algebraic meaning.
In 3D, where faces are spanned by two tangential vectors, we commit an abuse of notation and define the action of the tangent on $\vbf \in \mathbb{R}^3$ and $M \in \mathbb{R}^{3 \times 3}$ as 
	\begin{equation}\label{eq:tangentaction}
		\vbf \cdot \tbf = \vbf \times \nbf,		\quad\quad M \tbf  = [M_1^{\top} \times \nbf, M_2^{\top} \times \nbf, M_3^{\top} \times \nbf]^{\top},
	\end{equation}
where $M_i$ denotes row $i$ of $M$.

To avoid requiring $\mathcal{C}^1$-conforming piecewise polynomial spaces, the analysis revolves around a variant of the elliptic reconstruction operator~\cite{Makridakis:2003ws} and a Helmholtz decomposition.

	We define the elliptic reconstruction $\uc \in \V$ of the dG solution $\un \in \Vn$ to satisfy
\begin{equation} \label{elliptic-reconstruction}
\B(\uc, v) = \Bn (\un, v) \quad \forall v \in \V.
\end{equation}
The elliptic reconstruction is well defined for any $\un \in  \Vn$ due to the coercivity of~$\B(\cdot,\cdot)$. It equivalently satisfies
\begin{equation} \label{elliptic-reconstruction-bis}
	\int_\Omega \Dn^2 (\uc -\un) : \D^2 v = \int_\Omega \Lcal( \un ): \Dn^2 v \quad \forall v \in \V.
\end{equation}
We shall combine this with the following Helmholtz decomposition, which is shown in~\cite[Lemma~1]{BMRMorleyAPos} in 2D and~\cite[Lemma~5.2]{HuShi-MorleyApos} in 3D.
\begin{lem}[Helmholtz decomposition] \label{lemma:Helmholtz-dec}
Let $\Omega \subset \mathbb R^d$ and $\ssigmabold\in L^2(\Omega, \mathbb R^{d\times d})$.
There exist~$\xi \in H^2_0(\Omega)$, $\rho \in L^2_0(\Omega)$ or~$ (\rho_a, \rho_b, \rho_c)^{\top} \in L^2(\Omega, \mathbb R^{3})$, and $\Psibold_2 \in [H^1(\Omega)]^2$ or $\Psibold_3 \in H^1(\Omega, \mathbb R^{3\times 3})$, such that
\[
\ssigmabold = \D^2 \xi + \rhobold_d + \curlbf \,\Psibold_d, \quad\text{ where }\quad	\rhobold_2 = 
\begin{bmatrix}
0 & -\rho\\
\rho & 0 
\end{bmatrix},	\,\,
\rhobold_3 =
\begin{bmatrix}
0 & \rho_c & -\rho_b\\
-\rho_c & 0 & \rho_a \\
\rho_b & -\rho_a & 0 
\end{bmatrix}.
\]
Moreover, there exists a positive constant $c_\Omega$, depending only on $\Omega$, such that
\begin{equation} \label{bound:Helmholtz}
\Vert \D^2 \xi \Vert _{0,\Omega} + \Vert \rhobold_d \Vert_{0,\Omega} + \Vert \Psibold_d \Vert_{1,\Omega} \le c_{\Omega} \Vert \ssigmabold \Vert_{0,\Omega}.
\end{equation}
\end{lem}

Recall the following identity from \cite{abcd}:
\begin{equation} \label{magic:dG-formula}
		\sum_{\E \in \taun} \int_{\partial \E} \vbf\cdot \nbf \, \sigmabold = \int_{\En} \{ \vbf \cdot \nbf \} \llbracket \sigmabold \rrbracket + \int_{\EnI} \llbracket \vbf \cdot \nbf \rrbracket \{\sigmabold\} \quad
		\forall \vbf,\, \sigmabold \in [\Vn + V]^d.
\end{equation}	
	
	%%%%%%%
	\subsection{Upper bound} \label{subsection:reliability}
	%%%%%%%
Here, we show that the estimator of Definition~\ref{def:estimator} forms an upper bound of the error, using the elliptic reconstruction~\eqref{elliptic-reconstruction} and the Helmholtz decomposition of Lemma~\ref{lemma:Helmholtz-dec}.
For technical simplicity, we suppose that $f \in \Vn$.
More general cases may be treated by proceeding as in~\cite{cohen2012convergence,KV}, resulting in an additional data approximation term, which may dominate the estimator.
	
\begin{thm}[A posteriori error estimate for the $\h\p$-version dG scheme] \label{theorem:reliability}
Let $u \in H^2_0(\Omega)$ and $\un \in \Vn$ solve the biharmonic problem~\eqref{biharmonic-problem-weak} and the dG scheme~\eqref{IPdGFEM}, respectively,
and let $\eta$ be the error estimator of Definition~\ref{def:estimator}. 
Then, there exists a constant $c>0$, independent of~$\h$ and~$\p$, such that
\[
			\Vert u - \un \Vert^2_{dG} \le c \eta^2.
\]
\end{thm}
\begin{proof}
We use the elliptic reconstruction~$\uc$ in~\eqref{elliptic-reconstruction} to split the error into a \emph{conforming error} $\conferr := u - \uc$ and a \emph{nonconforming error} $\nonconferr := \uc - \un$, which we estimate separately.

%%%%%%
\paragraph*{\textbf{Estimate of conforming error~$\conferr$.}}
%%%%%%
Recalling \eqref{elliptic-reconstruction}, the elliptic reconstruction $\uc$ satisfies
\[
\Vert \D^2 (u-\uc) \Vert^2_{0,\Omega} 	= \B(u-\uc, u-\uc) = (\f,u-\uc)_{0,\Omega} - \Bn (\un, u-\uc),
\]
and the dG scheme~\eqref{IPdGFEM} implies that
\[
\Vert \D^2 \conferr \Vert^2_{0,\Omega}	=	(\f, \conferr - \vn)_{0,\Omega} - \Bn (\un, \conferr - \vn)  \quad \forall \vn \in \Vn.
\]
Choosing $\vn = \I \conferr$ and denoting $\etan := \ec - \I \conferr$, where $\I \conferr \in \Vn \cap \mathcal{C}^0(\Omega)$ is the quasi-interpolant from Corollary~\ref{proposition:BabuSuri-KarkulikMelenk} with~$s=2$, integrating by parts twice,
applying the definition~\eqref{lifting-operator} of the lifting operator $\Lcal$,	and using the continuity of $\etan$ in the relation~\eqref{magic:dG-formula} provide the error relation
\[
\begin{split}
				&\Vert \D^2 \conferr \Vert^2_{0,\Omega}	
				=  
				\int_\Omega (\f - \Deltan^2 \un) \etan - \int_{\EnI} \llbracket (\D^2 \un) \nbf \rrbracket \cdot \{\nabla \etan \} + \int_{\EnI} \llbracket \nbf \cdot \nabla \Delta \un \rrbracket  \etan \\
				& \quad \quad \quad \quad \quad \quad- \int_\Omega \Lcal (\un) : \Dn^2 \etan - \int_{\En} \taubold \llbracket \nabla \un \rrbracket \cdot \llbracket \nabla \etan \rrbracket=: \sum_{j=1}^5 T_j.
			\end{split}		
\]
We proceed with the estimate by treating each term~$T_j$ separately.

Estimate~\eqref{BabuSuri-KarkulikMelenk-prop} provides
\[
T_1 \leq 	\Big\Vert \Big(\frac{\hbf}{\pbf}\Big)^2 (f - \Deltan^2 \un) \Big\Vert_{0,\Omega}	\Big\Vert \Big(\frac{\hbf}{\pbf}\Big)^{-2} \etan \Big\Vert_{0,\Omega} \leq c \Big\Vert \Big(\frac{\hbf}{\pbf}\Big)^2 (f - \Deltan^2 \un) \Big\Vert_{0,\Omega} 
			\Vert \Dn^2 \conferr \Vert_{0,\Omega},
\]
and trace inequality~\eqref{standard:trace} and estimate~\eqref{BabuSuri-KarkulikMelenk-prop} give
\[
T_2 \leq c   	\Big\Vert   \Big( \frac{\hbf}{\pbf}\Big)^{\frac{1}{2}} \llbracket (\D^2 \un) \nbf \rrbracket \Big\Vert_{0,\EnI} \Vert  \D^2 \conferr  \Vert_{0,\Omega},	\quad\quad T_3	
\leq c 	\Big\Vert \Big( \frac{\hbf}{\pbf}\Big)^{\frac{3}{2}} \llbracket \nbf \cdot \nabla \Delta \un \rrbracket \Big\Vert_{0,\EnI} \Vert  \D^2 \conferr  \Vert_{0,\Omega}.
\]
Similarly, recalling Lemma~\ref{lemma:lifting-stab}, trace inequality~\eqref{standard:trace}, and estimate~\eqref{BabuSuri-KarkulikMelenk-prop}, we find that
\[
T_4 \leq c   (\Vert \sigmabold^{\frac{1}{2}} \llbracket \un \rrbracket \Vert^2_{0,\En} 	+ \Vert \taubold^{\frac{1}{2}} \llbracket \nabla \un \rrbracket \Vert^2_{0,\En})  \Vert  \D^2 \conferr \Vert_{0,\Omega}, \quad\quad
T_5	 			\leq 	c	\big\Vert \pbf^{\frac{1}{2}}  \taubold^{\frac{1}{2}} \llbracket \nabla \un \rrbracket \big\Vert_{0,\En} \Vert \D^2 \conferr \Vert_{0,\Omega}.\\
\]
Collecting the estimates for the individual terms $T_j$, we deduce that
\begin{equation} \label{bound:c-term}
\begin{split}
\Vert \D^2 \conferr \Vert_{0,\Omega} 	
& \le  c \Big(  \,\Big \Vert \Big(\frac{\hbf}{\pbf}\Big)^2 (\f - \Deltan^2 \un)  \Big\Vert_{0,\Omega} + \Big\Vert   \Big( \frac{\hbf}{\pbf}\Big)^{\frac{1}{2}} \llbracket (\D^2 \un) \nbf \rrbracket \Big\Vert_{0,\EnI} 	\\
& \qquad \quad + \Big\Vert \Big( \frac{\hbf}{\pbf}\Big)^{\frac{3}{2}} \llbracket \nbf \cdot \nabla \Delta \un \rrbracket \Big\Vert_{0,\EnI} + \big\Vert \pbf^{\frac{1}{2}}  \taubold^{\frac{1}{2}} \llbracket \nabla \un \rrbracket \big\Vert_{0,\En} 
+ \Vert \sigmabold^{\frac{1}{2}} \llbracket \un \rrbracket \Vert_{0,\En} \Big).
\end{split}
\end{equation}

%%%%%%
\paragraph*{\textbf{Estimate of nonconforming error $\nonconferr$.}}
%%%%%%
Applying the Helmholtz decomposition of Lemma~\ref{lemma:Helmholtz-dec} to $\ssigmabold = \Dn^2 \nonconferr$, the skew symmetry of term~$\rhobold$ implies that
\begin{equation} \label{split:AB}
				\Vert \Dn^2 \nonconferr \Vert^2_{0,\Omega} 	
				= \int_{\Omega} \Dn^2 \nonconferr : \D^2 \xi  + \int_{\Omega} \Dn^2 \nonconferr : \curlbf \, \Psibold.
\end{equation}
To estimate the first term of~\eqref{split:AB}, we use the smoothness of $\xi \in H^2_0(\Omega)$ with the property~\eqref{elliptic-reconstruction-bis},
Lemma~\ref{lemma:lifting-stab}, and the stability of the Helmholtz decomposition~\eqref{bound:Helmholtz} to find that
\begin{align} \label{bound:AII}
	\notag
	\int_{\Omega} \Dn^2 (\uc - \un) : \D^2 \xi	& \le \Vert \Lcal(\un) \Vert_{0,\Omega} \Vert \D^2 \xi \Vert_{0,\Omega} 
	\le (c_s)^{\frac12}(\Vert \sigmabold^{\frac{1}{2}} \llbracket \un \rrbracket \Vert^2_{0,\En} + \Vert \taubold^{\frac{1}{2}} \llbracket \nabla \un \rrbracket \Vert^2_{0,\En})^{\frac{1}{2}}   \Vert \D^2 \xi \Vert_{0,\Omega}
	\\& 
	\le c (\Vert \sigmabold^{\frac{1}{2}} \llbracket \un \rrbracket \Vert^2_{0,\En} + \Vert \taubold^{\frac{1}{2}} \llbracket \nabla \un \rrbracket \Vert^2_{0,\En})^{\frac{1}{2}} \Vert \Dn^2 (\uc - \un) \Vert_{0,\Omega}.
\end{align}
As for the second term of~\eqref{split:AB}, we insert the vector-valued version $\Ibf$ of the quasi-interpolant introduced in Corollary~\ref{proposition:BabuSuri-KarkulikMelenk} with $s=1$,
recall that $\Dn^2 = \nabla \nabla^{\top}$, integrate by parts twice and use properties of elementary differential operators to obtain
\begin{align}
\int_{\Omega} \Dn^2 \nonconferr : \curlbf  \, \Psibold 	&= \sum_{\E \in \taun} \Big( \int_{\partial \E} \big( (\Dn^2 \nonconferr) \tbf \big) : (\Psibold - \Ibf \Psibold) 	+ \int_{\partial \E} \nabla\nonconferr \cdot ((\curlbf\, \Ibf \Psibold) \nbf)    \Big), \label{bound:C-D}
\end{align}
observing our notational convention~\eqref{eq:tangentaction} for the tangential component  $(\Dn^2 \nonconferr) \tbf$ of the Hessian.

To estimate the first term of~\eqref{bound:C-D}, the relation~\eqref{magic:dG-formula}, the continuity of $\Psibold - \Ibf \Psibold$, and the fact that $\uc \in H^2_0(\Omega)$ give
\begin{align*}
			\sum_{\E \in \taun} 
				\int_{\partial \E} \big( (\Dn^2 \nonconferr) \tbf \big) : (\Psibold - \Ibf \Psibold) 
			&= 
			\int_{\En} \llbracket (\D^2 \nonconferr) \tbf \rrbracket : \{ \Psibold - \Ibf \Psibold  \}  
			= 
			- \int_{\En} \llbracket (\D^2 \un) \tbf \rrbracket : ( \Psibold - \Ibf \Psibold  ) 
			\\& 
			\leq 
			\Big \Vert  \Big( \frac{\hbf}{\pbf} \Big)^{\frac{1}{2}} \llbracket (\D^2 \un)\tbf \rrbracket \Big \Vert_{0,\En} 
			\Big \Vert  \Big( \frac{\pbf}{\hbf} \Big)^{\frac{1}{2}} (\Psibold - \Ibf \Psibold)      
			\Big \Vert_{0,\En}.
\end{align*}
Applying~\eqref{BabuSuri-KarkulikMelenk-prop}, and using~\eqref{inverse:H1L2} and the stability of the Helmholtz decomposition~\eqref{bound:Helmholtz} then imply
\begin{equation} \label{bound:CII}
			\begin{split}
				\widetilde c 
				\,
				\Big \Vert  \Big( \frac{\hbf}{\pbf} \Big)^{\frac{1}{2}} \llbracket (\D^2 \un) \tbf \rrbracket \Big \Vert_{0,\En} 
				\Vert \Psibold \Vert _{1,\Omega}
				\leq 
				c 
				\,
				\Big \Vert  \Big( \frac{\hbf}{\pbf} \Big)^{\frac{1}{2}} \llbracket (\D^2 \un) \tbf \rrbracket \Big \Vert_{0,\En}
				\Vert \Dn^2 \nonconferr \Vert_{0,\Omega}.
			\end{split}
\end{equation}
To estimate the second term of~\eqref{bound:C-D}, we note that $\llbracket (\curlbf \, \Ibf \Psibold) \nbf \rrbracket _{\EnI} = 0$ since each entry of $\Psibold$ is in $H^1(\Omega)$.
Consequently, each entry of $\Ibf \Psibold $ is also in $H^1(\Omega)$, implying $\curlbf \, \Ibf \Psibold \in H(\curl,\Omega)$;
see also~\cite{DariDuranPadraVampa-aposNC} and~\cite[proof of Lemma $3.3$]{CarstensenBartelsJansche-aposNC}.
		Combining this with~\eqref{magic:dG-formula} and $\uc \in H^2_0(\Omega)$, we find
\[
\sum_{\E \in \taun}  \int_{\partial \E} \! \nabla\nonconferr \cdot ((\curlbf\, \Ibf \Psibold) \nbf)      
 = \!-\! \int_{\En} \! \llbracket \nabla \un \rrbracket \cdot ( (\curlbf\, \Ibf \Psibold) \nbf ) \leq \Vert \taubold^{\frac{1}{2}} \llbracket \nabla \un \rrbracket \Vert_{0,\En}	\Vert \taubold^{-\frac{1}{2}} (\curlbf\, \Ibf \Psibold) \nbf \Vert_{0,\En}.
\]
Trace inverse estimate~\eqref{inverse:L2trace}, definition~\eqref{dG:stab-parameters} of~$\taubold$, the stability of~$\Ibf$, which follows from~\eqref{BabuSuri-KarkulikMelenk-prop}, and the Helmholtz decomposition further yield
\[
\Vert \taubold^{-\frac{1}{2}} (\curlbf\, \Ibf \Psibold) \nbf \Vert_{0,\En}	 \le c_1 \Vert \curlbf\, \Ibf \Psibold \Vert_{0,\Omega} \le c_2 \Vert \Ibf \Psibold \Vert_{1,\Omega}\le c_3 \Vert \Psibold \Vert_{1,\Omega} \le c_4 \Vert \Dn^2 \nonconferr \Vert_{0,\Omega}.
\]
Combined with estimates~\eqref{split:AB}, \eqref{bound:AII}, \eqref{bound:C-D}, and \eqref{bound:CII}, this provides
\begin{equation}\label{non-conforming:error:bound}
\Vert \Dn^2\nonconferr \Vert_{0,\Omega}\le c\Big(\Vert \sigmabold^{\frac{1}{2}} \llbracket \un \rrbracket \Vert_{0,\En}+\Vert \taubold^{\frac{1}{2}} \llbracket \nabla \un \rrbracket \Vert_{0,\En} + \Big\Vert  \Big( \frac{\hbf}{\pbf} \Big)^{\frac{1}{2}} \llbracket (\D^2 \un) \tbf \rrbracket \Big\Vert_{0,\En}
\Big).
\end{equation}
The result follows by combining~\eqref{non-conforming:error:bound} with~\eqref{bound:c-term}, and recalling definition~\eqref{dG:norm}.
\end{proof}

\begin{remark}[A comment on nonconforming error~$\nonconferr$]
Term~$\eta_{\E,5}$ in Definition~\ref{def:estimator} contains an additional factor $\pbf^{1/2}$, appearing through the estimate of $T_5$ above.
This suboptimal factor is present because the quasi-interpolation operator we use is only globally $\mathcal C^0$: while the jump terms involving $\un$ vanish, those involving $\nabla \un$ do not.
However, since the dG space does not generally contain a $\mathcal C^1$-conforming subspace with optimal approximation properties, it is not possible to improve this by constructing a $\mathcal C^1$-conforming quasi-interpolation operator.
\end{remark}

\begin{remark}[A comment on conforming error~$\conferr$]\label{averaging operator}
It is possible to derive an error estimate by splitting the error using an averaging operator to $\mathcal C^1$-conforming macro element or virtual element spaces, as in \cite{GeorgoulisHoustonIPdGhp,BrennerSung2005,brenner2019virtual}.
However, such an estimate for the nonconforming part of the estimator requires using $L_\infty$ to $L_2$ norm polynomial inverse inequalities several times.
This produces an error estimate, which is suboptimal in terms of the polynomial degree by~$\p^d$, rather than $\p$-optimal estimate~\eqref{non-conforming:error:bound} derived here.
\end{remark}

\begin{remark}[Hanging nodes] \label{handing-nodes}
Theorem~\ref{theorem:reliability} applies to meshes without hanging nodes. Hanging nodes in simplicial meshes may be removed using the well known red-green refinement strategy, although they cannot be removed in tensor product meshes without refining to the boundary.

The estimate for nonconforming error~$\nonconferr$ remains valid in presence of hanging nodes. This may be shown by arguing as in~\cite{submitted:Peter}.
The challenge is in constructing a $\mathcal C^0$-conforming quasi-interpolation operator for the estimate of conforming error~$\conferr$. 
For 2D parallelogram meshes with at most one hanging node per face, an explicit analysis of the $\p$-suboptimality may be performed as e.g. in~\cite[Theorem 4.72]{SchwabpandhpFEM} and~\cite[Theorem~3.6]{houston2000stabilized}.
The resulting estimate takes the form
\[
\Vert u - \un \Vert^2_{dG} \le c \sum_{\E \in \taun} \pE^2 \etaE^2.
\]
A similar result with additional suboptimality with respect to~$\p$ may be shown on cubic meshes by using~\cite[Section 6]{Schotzau:2018vg}. 
We further explore the influence of hanging nodes numerically in Section~\ref{section:nr}.
\end{remark}
	
%%%%%%%
\subsection{Local lower bound} \label{subsection:efficiency}
%%%%%%%
We show that the error estimator of Definition~\ref{def:estimator} provides a local lower bound on the error of the scheme measured in the dG norm.
The constant in the bound is optimal with respect to the mesh size~$\hbf$, but algebraically suboptimal in terms of the polynomial degree~$\pbf$.
	
\begin{thm}[Local lower bound] \label{theorem:efficiency}
Define the piecewise constant function~$\cetaEtw(\pbf)$ by
\begin{equation} \label{cetaEth}
\cetaEtw(\pbf)_{|\E}  =
\begin{cases}
\pbf_{|\E}^{6} & \text{if } \E \text{ is a triangle/tetrahedron}\\
\pbf_{|\E}^{4} & \text{if } \E \text{ is a parallelogram/parallelepiped}
\end{cases}\quad \quad \forall \E \in \taun,
\end{equation}
and, for $\E \in \taun$, let $\omegaE$ denote the patch of elements sharing a face with $\E$.

Let~$u \in H^2_0(\Omega)$ and $\un \in \Vn$ be the solutions to the biharmonic problem~\eqref{biharmonic-problem-weak} and the dG scheme~\eqref{IPdGFEM}, respectively.
Let~$\etaE$ be the local error estimator from Definition~\ref{def:estimator}.
Then, there exists a constant $c>0$ independent of $\hbf$ and $\pbf$ such that, for each $\E \in \taun$,
\begin{equation} \label{efficiency}
\etaE	 \le c \big( \big\Vert \pbf^{4d+\frac{3}{2}} \cetaEtwsq(\pbf)  \Dn^2(u-\un) \big\Vert_{0,\omegaE} + \big\Vert \pbf^{\frac{3}{2}} \cetaEtw(\pbf) \taubold ^{\frac{1}{2}} \llbracket \nabla \un \rrbracket \big\Vert_{0,\partial \E} 
														+ \big\Vert \sigmabold ^{\frac{1}{2}} \llbracket \un \rrbracket \big\Vert_{0,\partial \E}\big).
\end{equation}
\end{thm}
\begin{proof}
Without loss of generality, we suppose that the element $\E$ has diameter $\hE = 1$. The quasi-uniformity assumption~\eqref{quasi-uniformity:assumption} ensures that the diameter of each face is also approximately $1$.
The proof in the general case was shown in a similar setting in~\cite{Virtanen:Apos-biharmonic}, and follows by a scaling argument.

The terms~$\eta_{\E, 5}$ and~$\eta_{\E, 6}$ form part of the dG norm, up to a scaling by~$\pbf$ and are therefore estimated trivially.
		
\paragraph*{\textbf{The error equation.}}
The biharmonic problem~\eqref{biharmonic-problem-weak} and the dG scheme~\eqref{IPdGFEM} imply that
\[
\Bn(u-\un,v) = (\f,v)_{0,\Omega} - \Bn(\un, v) \quad \quad \forall v \in H^2_0(\Omega),
\]
and integrating by parts twice produces the error equation
\begin{equation} \label{error:equation}
					(\Dn^2(u-\un), \Dn^2 v)_{0,\Omega} 
					 = \!\!\! \sum_{\E \in \taun} \!\! \int_{\E} \! (\f - \Delta^2 \un)v - \!\!\! \sum_{\F \in \En} \!\! \int_{\F } \! ( \llbracket (\D^2 \un)\nbf \rrbracket \cdot \{ \nabla v \} \! - \! \llbracket \nbf \cdot \nabla \Delta \un \rrbracket \{ v \}  ).
\end{equation}
		
%%%%%%%%%%
\paragraph*{\textbf{Estimate of $\eta_{\E,1}$.}}
		Denote the bubble function on $\E$ by $\bE$. 
		Applying~\eqref{inverse:bubble} produces
		\[
		\Vert \f - \Delta^2 \un \Vert^2_{0,\E}	 \le c  \pE^{4d} \Vert  \bE (\f - \Delta^2 \un) \Vert^2_{0,\E} 
		=
		c \pE^{4d}  ( \bE^2 (\f - \Delta^2 \un), \f - \Delta^2 \un)_{0,\E},
		\]
		and since $\bE^2 (\f - \Delta^2 \un)\in V$ and $\bE{}{}_{|\partial \E}=(\nbf \cdot \nabla \bE{}){}_{|\partial \E}=0$, error equation~\eqref{error:equation} implies
		\[
		\begin{split}
			(  \f - \Delta^2 \un, \bE^2(\f - \Delta^2 \un))_{0,\E} 	
			& = (\D^2 (u-\un) , \D^2(\bE^2 (\f - \Delta^2 \un)))_{0,\E}. \\
		\end{split}
		\]
		Applying the inverse inequality~\eqref{inverse:weighted-H1L2} twice, 
		it follows that
		\begin{equation} \label{bound:internal-residual}
			\begin{split}
				p_{\E}^{-2}\Vert \f - \Delta^2 \un \Vert_{0,\E}	 	
				& 
				\le 
				c \pE^{4d} \Vert \D^2(u-\un) \Vert_{0,\E}.
			\end{split}
		\end{equation}
		
		%%%%
		\paragraph*{\textbf{Estimates of $\eta_{\E,4}$.}}
		The quasi-uniformity assumptions~\eqref{quasi-uniformity:assumption} and~\eqref{bounded:polynomial} on $\hbf$ and $\pbf$, alongside the inverse inequality~\eqref{inverse:H1L2}, imply that $\eta_{\E,4}$ may be estimated by
\[
			\eta_{\E,4}^2 
			=  
			\frac{1}{2} \sum_{\F \in \EE}  
			\big\Vert {\pbf}^{-\frac{1}{2}} \llbracket (\D^2 \un) \tbf \rrbracket \big\Vert^2_{0,\F} 
			\le  
			c \sum_{\F \in \EE}  \big\Vert {\pbf}^{\frac{3}{2}} \llbracket  \nabla \un \rrbracket \big\Vert^2_{0,\F} \leq c \eta_{\E,5}^2.
\]
		
		%%%%
\paragraph*{\textbf{Estimate of $\eta_{\E,3}$.}}
We split $\eta_{\E,3}$ into the orthogonal normal and tangential components on each face $\F \in \EE \cap \EnI$, giving
\begin{equation} \label{bound:eta3-split}
\big\Vert {\pbf}^{-\frac{1}{2}} \llbracket (\D^2 \un) \nbf \rrbracket \big\Vert^2_{0,\F} 	= \big\Vert {\pbf}^{-\frac{1}{2}} \llbracket \nbf^{\top} (\D^2 \un) \nbf \rrbracket \big\Vert^2_{0,\F} + \big\Vert  {\pbf}^{-\frac{1}{2}} \llbracket \nbf^{\top} (\D^2 \un) \tbf \rrbracket \big\Vert^2_{0,\F},
\end{equation}
due to the symmetry of $\D^2 \un$ and recalling our notational convention~\eqref{eq:tangentaction} for its tangential component.
Arguing as before, we bound the tangential component using the inverse inequality~\eqref{inverse:H1L2} and the quasi-uniformity assumptions~\eqref{quasi-uniformity:assumption} and~\eqref{bounded:polynomial}, producing
\begin{equation} \label{bound:eta3-0}
			\big\Vert  {\pbf}^{-\frac{1}{2}} \llbracket \nbf^{\top} (\D^2 \un) \tbf \rrbracket \big\Vert^2_{0,\F} 
			\leq 
			c  
			\big\Vert {\pbf}^{\frac{3}{2}} \llbracket \nbf \cdot \nabla \un \rrbracket \big\Vert^2_{0,\F} 
			\leq 
			c 
			\eta^2_{\E,5}.
\end{equation}
To bound the normal component of~\eqref{bound:eta3-split}, suppose that $\F = \overline{\E} \cap \overline{\E}^{*}$ for some $\E^{*} \in \taun$.
We construct a bespoke bubble function $b_{\F}$ on $\F$ using the kite $\Etilde \subset \overline{\E} \cup \overline{\E}^{*}$ associated with~$\F$ discussed in Section~\ref{subsection:meshes-degrees}.
We begin with the standard face bubble function $\bEtilde$ on $\F$ in $\Etilde$, defined as the product of the nodal linear basis functions associated with the vertices of $\F$ on each triangle forming $\Etilde$.
This satisfies $\llbracket \nabla \bEtilde \cdot \nbf  \rrbracket_{F} = 0$ due to the symmetry of the kite~$\Etilde$.
Let~$\bell$ denote an affine function such that $\bell{}_{|\F} = 0$ and $(\nabla \bell\cdot \nbfF){}_{|\F} = 1$.
Define $b_\F = \bell{}\bEtilde^2$ on $\Etilde$ and $b_\F = 0$ otherwise and observe that
\[
			\begin{split}
				&b_\F\in \mathcal C^1(\Omega) \cap H^2_0(\Omega),
				\text{ so } 
				 \llbracket b_\F \rrbracket_{\Fhat} 
				 =  
				 \{  b_\F \}_{\Fhat} 
				 =
				 0,\quad
				 \llbracket \nabla b_\F \rrbracket_{\Fhat} = \bold{0} 
				 \quad \forall \Fhat \in \En, 
				\\
				&\qquad\quad
				\{ \nabla b_\F\} _{\Fhat} =  0  \text{ for all $\hat{\F}  \in \En\ \backslash \F$},	 
				\text{ with }
				\{ \nabla b_\F \}_{\F} \cdot \nbf_\F =  \bEtilde^2{}_{|\F}.
			\end{split}
\]
With $\Eit$ denoting the extension operator of Corollary~\ref{corollary:extension-inverse}, we introduce
\begin{equation} \label{definition:r}
		v = b_\F r
		\quad\text{ with }\quad
		r = \Eit \left( \pF^{-1}  \llbracket \nbf^{\top} (\D^2 \un) \nbf \rrbracket_{\F}   \right).
\end{equation}
The fact that~$v \in H^2_0(\Omega)$ and error equation~\eqref{error:equation} imply that
\begin{align} \notag
\int_{\F} \llbracket (\D^2 \un)\nbf \rrbracket \cdot \{ \nabla v \}	
& = 
(\f - \Deltan^2 \un, v)_{0,\Etilde} - (\Dn^2(u-\un), \Dn^2 v)_{0,\Etilde} 
\\
& \le \Vert \f - \Deltan^2 \un \Vert_{0,\Etilde} \Vert v \Vert _{0,\Etilde} + \Vert \Dn^2(u-\un)\Vert _{0,\Etilde} \Vert \Dn^2  v \Vert_{0,\Etilde},
\label{bound:eta3-1}
\end{align}
and we estimate the terms on the right-hand side of~\eqref{bound:eta3-1} separately.
Applying~\eqref{eq:extensionoperator}, we have	
\begin{equation} \label{bound:v}
\Vert v \Vert_{0,\Etilde} 	\le c \left \Vert \Eit\left( \pbf^{-1}  \llbracket \nbf^{\top} (\D^2 \un) \nbf \rrbracket \right) \right \Vert_{0,\Etilde} 
\le c  \left \Vert \pbf^{-2} \llbracket \nbf^{\top} (\D^2 \un) \nbf \rrbracket \right\Vert_{0,\F}.
\end{equation}
By the definition of $v$, we obtain
\[
\Vert \Dn^2 v \Vert_{0,\Etilde} \le \Vert r \Deltan b_{\F}  \Vert_{0,\Etilde} + 2  \Vert \nablan b_{\F} \cdot \nablan r \Vert_{0,\Etilde} +  \Vert b_{\F} \Deltan r \Vert_{0,\Etilde} \le c\big(   \Vert r \Vert_{0,\Etilde}  +  \Vert \nablan r \Vert_{0,\Etilde} + \Vert \Deltan r \Vert_{0,\Etilde}  \big),
\]
and estimate~\eqref{eq:extensionoperator} and definition~\eqref{definition:r} of~$r$ provide
		\begin{equation} \label{bound:D2v}
			\Vert \Dn^2 v \Vert_{0,\Etilde} \le  c     \left\Vert  \pbf^{2}  \llbracket \nbf^{\top} (\D^2 \un) \nbf \rrbracket   \right\Vert_{0,\F}  .
		\end{equation}
Recalling cell residual bound~\eqref{bound:internal-residual}, and equations~\eqref{bound:v} and~\eqref{bound:D2v}, relation~\eqref{bound:eta3-1} yields
		\begin{equation} \label{bound:eta3-1.5}
			\begin{split}
				&\int_{\F} \llbracket (\D^2 \un)\nbf \rrbracket \cdot \{ \nabla v \} 
				\le c  
				 \left\Vert  \pbf^{4d} \Dn^2(u-\un) \right\Vert _{0,\Etilde}
				\left\Vert   \llbracket \nbf^{\top} (\D^2 \un) \nbf \rrbracket   \right\Vert_{0,\F}.
			\end{split}
		\end{equation}
We complete the estimate by showing that we can bound term~$\Vert {\pbf}^{-\frac{1}{2}} \llbracket \nbf^{\top} (\D^2 \un) \nbf \rrbracket \Vert^2_{0,\F}$ from above by the left-hand side of~\eqref{bound:eta3-1.5}.
Splitting $\{\nabla v\}$ into its normal and tangential components,
and using the fact that $\{v\}|_\F =\{(\nabla v) \cdot \tbf\}|_\F =0$, we have $\nabla v|_\F =  (\nbf \cdot \nabla v|_\F)\nbf$ and it follows that
\[
			\int_{\F} \llbracket (\D^2 \un)\nbf \rrbracket \cdot \{ \nabla v \}
			= \int_{\F} \llbracket 
			\nbf^{\top} (\D^2 \un) \nbf
			\rrbracket  \{\nbf \cdot \nabla v \} . 
\]
Applying
Proposition~\ref{proposition:inverse-bubble}, with $s = 3$ on simplicial elements and $s=2$ otherwise, produces
\[
\begin{split}
\big\Vert {\pbf}^{-\frac{1}{2}} \llbracket \nbf^{\top} (\D^2 \un) \nbf \rrbracket \big\Vert_{0,\F}^2&
				\le  
				c \int_{\F} \cetaEtw (\pbf) \llbracket \nbf^{\top} (\D^2 \un) \nbf \rrbracket  \bEtilde^2 r
				=  
				c \int_{\F} \cetaEtw (\pbf)\llbracket  (\D^2 \un) \nbf \rrbracket   \cdot \{ \nabla v \},
			\end{split}
\]
where~$\cetaEtw$ is defined in~\eqref{cetaEth}. Combined with~\eqref{bound:eta3-1.5}, this produces
\[
			\big \Vert \pbf^{-\frac{1}{2}}  \llbracket \nbf^{\top} \D^2 \un \nbf \rrbracket \big\Vert_{0,\F} 		
			\le  
			c  
			\big\Vert \pbf^{4d+\frac{1}{2}} \cetaEtw (\pbf)    \Dn^2 (u-\un) \big\Vert_{0,\E \cup \E^*}.
\]
Recalling~\eqref{bound:eta3-0} and~\eqref{bound:eta3-split}, we obtain the estimate
		\begin{equation} \label{bound:eta3-final}
			\big \Vert \pbf^{-\frac{1}{2}}  \llbracket \nbf^{\top} \D^2 \un
			\rrbracket \big\Vert_{0,\F} 
			\le  
			c  
			\big \Vert \pbf^{4d+\frac{1}{2}} \cetaEtw (\pbf)    \Dn^2 (u-\un) \big \Vert_{0, \E \cup \E^*}  
			+ 
			c  
			\big\Vert {\pbf}^{\frac{3}{2}} \llbracket  \nbf \cdot \nabla \un \rrbracket \big\Vert_{0,\F}.
		\end{equation}
The final bound on~$\eta_{\E,3}$ follows by summing over all the nonboundary faces of~$\E$.
		
		%%%%
		\paragraph*{\textbf{Estimate of $\eta_{\E,2}$.}}
		Once again, let $\Etilde$ be the kite associated with $\F$, discussed in Section~\ref{subsection:meshes-degrees}, and let $\bEtilde$ be the bubble on $\Etilde$ associated with the face $\F$, introduced above.
		Let
		\begin{equation}\label{eq:eta2TestFunction}
		v = \bEtilde^2 \Eit\left( \pbf^{-3} \llbracket \nbf \cdot \nabla \Delta \un \rrbracket \right),
		\end{equation}
which, due to the properties of the bubble function, the trace inequality~\eqref{standard:trace}, and stability estimates~\eqref{eq:extensionoperator}, satisfies
		\begin{align*}
			\Vert \pbf^{4} v \Vert_{0,\Etilde} 
			+ 
			\Vert \pbf^{2} \nablan v \Vert_{0,\Etilde} 
			+ 
			\Vert \Dn^2 v \Vert_{0,\Etilde} 
			&\le 
			c \left \Vert     \llbracket \nbf \cdot \nabla \Delta \un  \rrbracket \right\Vert_{0,\F} .
		\end{align*}
Selecting~$v$ as the test function in error equation~\eqref{error:equation}, we deduce that
		\begin{align} \label{initial-bound:etaE2}
				\int_\F \llbracket \nbf \cdot \nabla \Delta \un \rrbracket v 
				& = (\f - \Delta^2 \un, v)_{0,\Etilde} - (\Dn^2(u-\un), \Dn^2 v)_{0,\Etilde} + \int_{\F} \llbracket (\D^2 \un) \nbf \rrbracket \cdot \{ \nabla v \} \\
				& \le \Vert \f - \Delta^2 \un \Vert_{0,\Etilde} \Vert v \Vert _{0,\Etilde} + \Vert \Dn^2(u-\un)\Vert _{0,\Etilde} \Vert \Dn^2 v \Vert_{0,\Etilde} + \Vert \llbracket  (\D^2 \un) \nbf \rrbracket \Vert_{0,\F} \Vert \nabla v \Vert_{0,\F}.
				\notag
		\end{align}
Recalling bound~\eqref{bound:internal-residual} on the cell residual and that on $\eta_{\E,3}$ from~\eqref{bound:eta3-final}, combined with the trace inequality \eqref{standard:trace}, it follows that
\[
			\begin{split}
				\int_\F \llbracket \nbf \cdot \nabla \Delta \un \rrbracket v 
				\le  c  \left( \big\Vert \pbf^{4d} \cetaEtw (\pbf)  \Dn^2(u- \un)  \big\Vert_{0, \Etilde} + \big\Vert \pbf \llbracket  \nbf \cdot \nabla \un \rrbracket \big\Vert_{0,\F} \right)   \big\Vert  \llbracket \nbf \cdot \nabla \Delta \un \rrbracket \big\Vert_{0,\F}  . \\
			\end{split}
\]
		Using Proposition~\ref{proposition:inverse-bubble} with $s =3$ on simplicial elements and $s=2$ otherwise, we find
\[
			\begin{split}
				\int_\F \pbf^{-3}  \llbracket \nbf \cdot  \nabla \Delta \un \rrbracket^2  	
				& 
				\le 
				c 
				\int_\F \cetaEtw (\pbf) \llbracket \nbf \cdot  \nabla \Delta \un \rrbracket v,
			\end{split}
\]
		where $\cetaEtw$ is defined in~\eqref{cetaEth}, and therefore
\[
			\begin{split}
				\big \Vert \pbf^{-\frac{3}{2}} \llbracket \nbf \cdot \nabla \Delta \un  \rrbracket \big\Vert _{0,\F} 
				& 
				\le 
				c  
				\big\Vert \pbf^{4d+\frac{3}{2}} \cetaEtwsq(\pbf)   \Dn^2(u- \un)    \big\Vert_{0, \Etilde}  
				+  
				c
				\big\Vert \pbf^{\frac{5}{2}} \cetaEtw(\pbf) \llbracket  \nbf \cdot \nabla \un \rrbracket \big\Vert_{0,\F}.
			\end{split}
\]
The bound on $\eta_{\E,2}$ follows by summing over all nonboundary faces of $\E$, and the theorem is proven combining the bounds on the individual estimator terms.
\end{proof}
	
	\begin{remark}[Improved suboptimality on certain meshes] \label{remark:gain-in-p}
		In 2D, and when employing particular meshes in 3D, the suboptimality in~\eqref{efficiency} with respect to the polynomial degree can be reduced.
		If the kite $\Etilde$ constructed on each face $\F$ can be replaced by a rhombus in 2D or a rhomboidal polyhedron in 3D,
		the additional symmetry implies that
		the function $v$ constructed in~\eqref{eq:eta2TestFunction} further satisfies $\nabla v  {}_{|\F} = \bold{0}$, implying the last term on the right-hand side~\eqref{initial-bound:etaE2} vanishes.
Bound~\eqref{efficiency} therefore becomes
\[
			\etaE
			\le 
			c 
			\big( 
			\big\Vert \pbf^{4d+\frac{1}{2}} \cetaEtw(\pbf)   \Dn^2(u-\un) \big\Vert_{0,\omegaE}  
			+ 
			\big\Vert \pbf^{\frac{1}{2}}  \taubold ^{\frac{1}{2}} \llbracket \nabla \un \rrbracket \big\Vert_{0,\partial \E} 	
			+ 
			\big\Vert \sigmabold ^{\frac{1}{2}} \llbracket \un \rrbracket \big\Vert_{0,\partial \E}
			\big).
\]
Such a choice for $\Etilde$ is always possible in 2D and on parallelepiped meshes in 3D, and on tetrahedral meshes in certain circumstances.
In general, in the latter case, we have~\eqref{efficiency}.
\end{remark}

\begin{remark}[Application to $\mathcal{C}^0$-interior penalty methods] \label{remark:basta}
The same arguments may be used to prove upper and lower bounds for the estimator for $\mathcal C^0$-interior penalty methods, with the difference that $\llbracket \un \rrbracket = \llbracket (\nabla v)\cdot\tbf \rrbracket = 0$.
Term~$\eta_{\E,6}$ in Definition~\ref{def:estimator} would therefore vanish, whereas term~$\eta_{\E,5}$ would become $\llbracket \nbf \cdot \nabla\un \rrbracket$.
\end{remark}
	
\begin{remark}[Inhomogeneous boundary data]\label{Inhomogous-BC}
Inhomogeneous Dirichlet boundary conditions may be treated similarly, as in~\cite{beirao2010nonhomo}.
Suppose that~$\gD$ and~$\gN$ are the two Dirichlet boundary conditions of~$u \in H^2(\Omega)$ over~$\partial \Omega$.
In particular,  we have $u{}_{|\partial \Omega} = \gD \in H^{1-\varepsilon} (\partial \Omega)$ for all~$\varepsilon >0$ and, for all the faces~$\FB$ of~$\partial \Omega$, $\gD{}_{|\FB} \in H^{\frac{3}{2}}(\FB)$.
Moreover, $\nbf \cdot \nabla  u{}_{|\partial \Omega} = \gN \in H^{-\varepsilon} (\partial \Omega)$ for all~$\varepsilon >0$ and, for all the faces~$\FB$ of~$\partial \Omega$, $\gN{}_{|\FB} \in H^{\frac{1}{2}}(\FB)$.
The dG scheme \eqref{IPdGFEM} then reads: find $\un \in \Vn$ such that $\Bn(\un, \vn) = l(\vn)$ for all $\vn \in \Vn$, where $l(\vn)$ is defined as
\[
			l(\vn) = (\f, \vn)_{0,\Omega}
			+\int_{\partial \Omega} \big(\gD (\sigmabold \vn +\nbf \cdot \nabla \Delta \vn ) \big)
			+ \big((\gN \nbf + ((\nabla \gD) \cdot \tbf) \tbf) \cdot (\taubold \nabla \vn - (\D^2 \vn) \nbf )\big).
\]
The boundary contributions in the jump terms appearing in $\eta_{\E,4}$, $\eta_{\E,5}$, and $\eta_{\E,6}$ of the error estimator, defined in Definition~\ref{def:estimator}, become
\[
\begin{split}
					& \llbracket (\D^2 \un) \tbf \rrbracket{}_{|\F \in \EnB}
					= 
					\big( \tbf^{\top} \D^2 (\un - \gD) \tbf \big) \tbf 
					+\big(  \nbf^{\top} ((\D^2\un) \tbf - \nabla \gN) \big) \nbf,\\
					&\llbracket \nabla \un  \rrbracket{}_{|\F \in \EnB}= 
					\big(\nbf \cdot  \nabla \un - \gN \big)\nbf 
					+ \big( (\nabla(\un - \gD)) \cdot \tbf \big) \tbf, \qquad 
					 \llbracket \un  \rrbracket{}_{|\F\in \EnB} =  \un - \gD    .\\
				\end{split}
\]
The reason why we need to pick the nonhomogeneous boundary conditions in spaces with low regularity on~$\partial \Omega$ is that~$\Omega$ is Lipschitz,
and the trace theorems for Lipschitz domains have bounded regularity shift properties; see e.g. \cite[Theorem~A.20]{SchwabpandhpFEM}.
\end{remark}
	
	%%%%%%%%%%%%%%%%%%%%%%%%%%%%%%%%%%%%%%%%%%%%%%%%%%%%%%%%%%%%%%%%%%%%%%%%%%%
	\section{Numerical results} \label{section:nr}
	%%%%%%%%%%%%%%%%%%%%%%%%%%%%%%%%%%%%%%%%%%%%%%%%%%%%%%%%%%%%%%%%%%%%%%%%%%%
In this section, we present some numerical experiments assessing the performance of the dG scheme~\eqref{IPdGFEM} and the estimator $\eta$ from Definition~\ref{def:estimator}.
The results in this section were computed using the AptoFEM library, developed by Professor P. Houston and collaborators.

We take the dG parameters $c_{\sigmabold} = c_{\taubold} = 10$,
and use the estimator to drive~$\h$- and $\h\p$-adaptive algorithms based on the standard iteration
\[
\text{solve} \quad \longrightarrow \quad \text{estimate} \quad \longrightarrow \quad \text{mark} \quad \longrightarrow \quad \text{refine}.
\]
For $\h$-adaptivity, the marking step uses the maximum strategy with parameter $\theta = 0.5$, and refinement in $d$ dimensions is achieved by splitting marked mesh elements into $2^d$ child elements of the same type.
Hanging nodes may be eliminated from triangular or tetrahedral meshes using the well known red-green refinement algorithm.

The $\h\p$-adaptive algorithm employs a variant of the Melenk-Wohlmuth marking strategy~\cite[Section~4.2]{MelenkWohlmuth_hpFEMaposteriori}
to determine whether to refine an element by splitting it into children or by increasing the local polynomial degree.
This is presented in Algorithm~\ref{algorithm:hp:refinement:strategy}, and we take the parameters~$\sigma = 0.7$, $\gamma_\h = 3$, $\gamma_\p = 0.9$,
and the initial predicted error indicator on each element is taken to be infinite to ensure the algorithm initially attempts to increase~$\p$.
Here, we modify the algorithm by coarsening in~$\h$ when increasing~$\p$ and vice versa.

In both cases, we allow~$\p$ varying by at most one across each face, and enforce a maximum of one hanging node per face.

\begin{algorithm}
\caption{The marking algorithm for $\h\p$-adaptivity on adaptive step $n$.}
\label{algorithm:hp:refinement:strategy}
\begin{algorithmic}
	\State
	$
	\text{\textbf{Inputs:}}
	\begin{cases}
		\text{Mesh and discrete solution: } &\taun, \,\un,
		\\
		\text{Computed and predicted error indicators: } & \eta_{\text{comp},n}, \, \eta_{\text{pred},n}
		\\
		\text{Marking parameters: } & \sigma, \, \gamma_\h, \, \gamma_\p, 
	\end{cases}
	$
\For {$\E \in \taun$}
\If{$\eta_{\text{comp}, \E, n}^2 \ge \sigma \max_{\E} \eta_{\text{comp}, \E, n}^2$}
				\Comment{\emph{mark $\E$ for refinement}}
\If{$\eta_{\text{comp}, \E, n}^2 \ge  \eta^2_{\text{pred},\E,n}$}
				\Comment{\emph{$\h$-refinement}}
\State{subdivide $\E$ into $N^\E$ children $\E_S$}
\State degree of accuracy on all $\E_S$ decreased by $1$
\State $\eta_{\text{\text{pred}}, \E_S, n+1}^2 = \frac{\gamma_\h}{N^\E} \, (0.5)^{2\pE - 2} \eta^2_{\text{comp}, \E, n}$
\Else
				\Comment{\emph{$\p$-refinement}}
\State degree of accuracy on $\E$ increased by $1$
\State mark $\E$ for $\h$-coarsening
\State $\eta^2_{\text{pred},\E,n+1} = \gamma_\p\, \eta^2_{\text{comp},\E,n}$
\EndIf
\Else{}
				\Comment{\emph{no refinement}}
\State $\eta^2_{\text{pred},\E,n+1} = \eta^2_{\text{pred},\E,n}$
\EndIf
\EndFor
\end{algorithmic}
\end{algorithm}

We focus on the following two particular benchmark problems.

%%%
\paragraph*{\textbf{L-shaped domain benchmark.}}
Let $\Omega_1 = (-1,1)^2 \setminus \big( [0,1)\times (-1,0] \big)$ and let $(r, \theta)$ denote the polar coordinates centered at the re-entrant corner $(0,0)$.
We take the benchmark solution
\begin{equation} \label{u1}
u_1(r, \theta) 	=  r^{\frac{4}{3}} \sin \Big( \frac{4}{3} \theta  \Big)\in H^{\frac{7}{3}-\varepsilon}(\Omega_1) \quad \forall	\varepsilon > 0,
\end{equation}
which satisfies the biharmonic problem on this domain with $f = 0$ and inhomogeneous boundary data, meaning that we modify the error estimator according to Remark~\ref{Inhomogous-BC}.
	
%%%
\paragraph*{\textbf{3D benchmark.}}
On the domain $\Omega_2 = [0,1]^3$, we consider the smooth benchmark solution
\begin{equation} \label{u2}
	u_2(x, y, z) = \big( \sin(\pi x) \sin(\pi y) \sin(\pi z) \big)^2,
\end{equation}
which satisfies the inhomogeneous biharmonic problem with homogeneous boundary conditions.

%%%%%%%
\subsection{Dependence of the effectivity index on~$p$} \label{subsection:effectivity-index-p}
%%%%%%%
The results of Theorems~\ref{theorem:reliability} and~\ref{theorem:efficiency} differ by an algebraic function of the polynomial degree~$\p$.
To understand the practical manifestation of this gap, we investigate the dependence on~$\p$ of the
effectivity index
\[
	\text{effectivity} = \frac{\eta}{\Vert u-\un \Vert_{dG}}.
\]
Figure~\ref{figure:effectivity-index} shows the effectivity index for the L-shaped domain benchmark $u_1$ in~\eqref{u1}
using a mesh of 12 square elements as the uniform polynomial degree is varied from $2$ to $20$.
We observe an algebraic growth like $\p^{1.8}$, which is significantly less severe than those predicted in Theorem~\ref{theorem:efficiency} and discussed in Remark~\ref{remark:gain-in-p}.

\begin{figure}
	\centering
	\includegraphics [angle=0, width=0.40\textwidth]{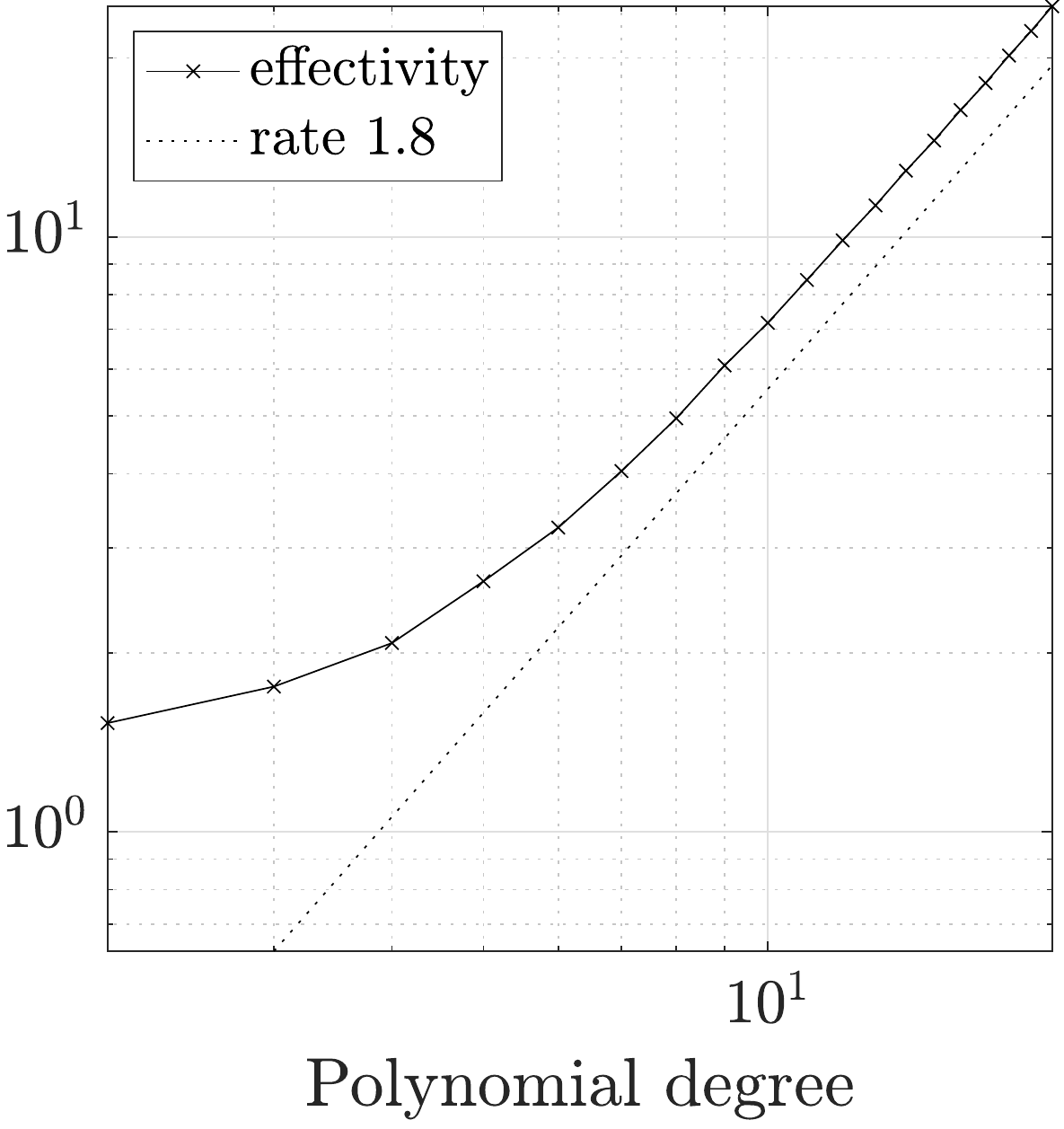}
\caption{Effectivity index for the exact solution $u_1$ defined in~\eqref{u1} on a uniform mesh consisting of $12$ square elements and polynomial degrees 2 to 20.}
\label{figure:effectivity-index}
\end{figure}
	
%%%%%%%
\subsection{$\h$- and $\h\p$-adaptivity for the L-shaped domain benchmark $u_1$ in~\eqref{u1}} \label{subsection:nr-hp}
%%%%%%%

We study the behaviour of the estimator through the $\h$- and $\h\p$-adaptive algorithms described above.
Figure~\ref{figure:nr-u2} compares the errors, estimators and effectivities for the $\h$-adaptive algorithm with fixed $\p=2$ and~$3$ against those of the $\h\p$-adaptive algorithm initialised with~$\p = 2$, using square meshes in both cases.
We observe that the $\h$-adaptive algorithm recovers the optimal rate of convergence of $-\frac{1}{2}$ for $\p=2$, and~$-1$ for~$\p=3$, with respect to the number of degrees of freedom, despite the singularity at the re-entrant corner.
This is reflected in the effectivity indices, which remain at an approximately constant value between 1.5 and 2 throughout the simulation, demonstrating excellent agreement between the error and estimator.

As expected, the $\h\p$-adaptive algorithm converges more rapidly with respect to the number of degrees of freedom, although the effectivity index may be observed to grow throughout the simulation,
due to the dependence of the effectivity on~$p$, as explored in Section~\ref{subsection:effectivity-index-p}.

We also investigate the impact of the mesh geometry and the presence of hanging nodes on the estimator.
The results are plotted in Figure~\ref{figure:nr-u2-2}, for the $\h$-adaptivity with $p=3$ and meshes consisting of square elements with hanging nodes, triangular elements with hanging nodes, or triangular elements without hanging nodes.
The scheme and estimator appear to be robust with respect to the choice of mesh, with similar results recovering optimal convergence rates in all cases.
The effectivities are slightly higher on triangular elements, between approximately 1.5 and 4, and appear to be slightly increased by removing hanging nodes.

Examples of the adapted meshes produced by the two algorithms are shown in Figure~\ref{figure:hp-meshes}.
For the $\h\p$-adaptive algorithm, we colour the mesh using the polynomial degree on each element.
We observe the expected grading behaviour, with fine elements with low polynomial degrees placed around the singularity at the re-entrant corner, and large elements with high polynomial degrees elsewhere.
Moreover, the meshes reflect the symmetry of the problem in the line $y=-x$.

%%%
\begin{figure}
	\centering
	\includegraphics [angle=0, width=0.49\textwidth]{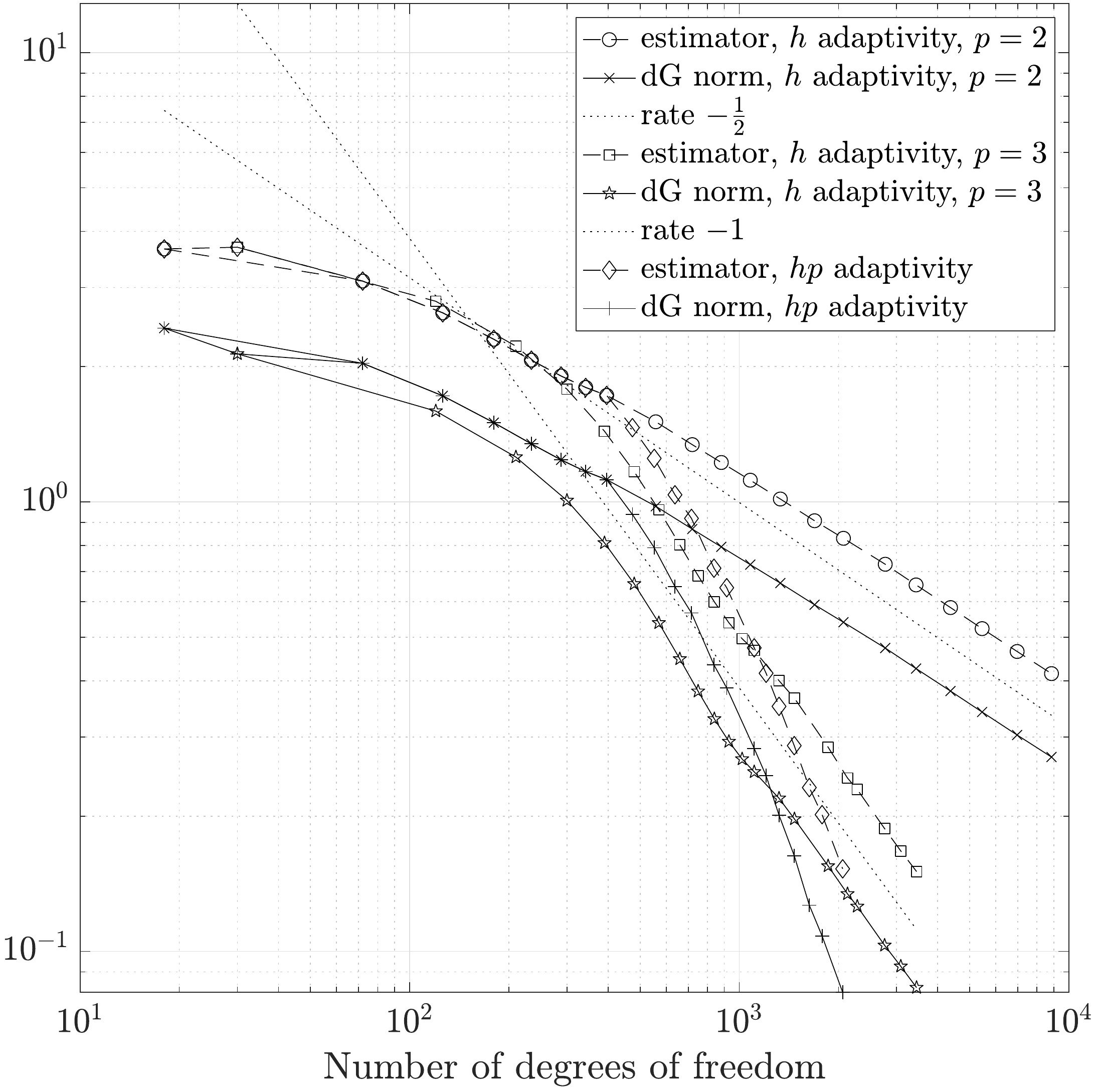}
	\includegraphics [angle=0, width=0.48\textwidth]{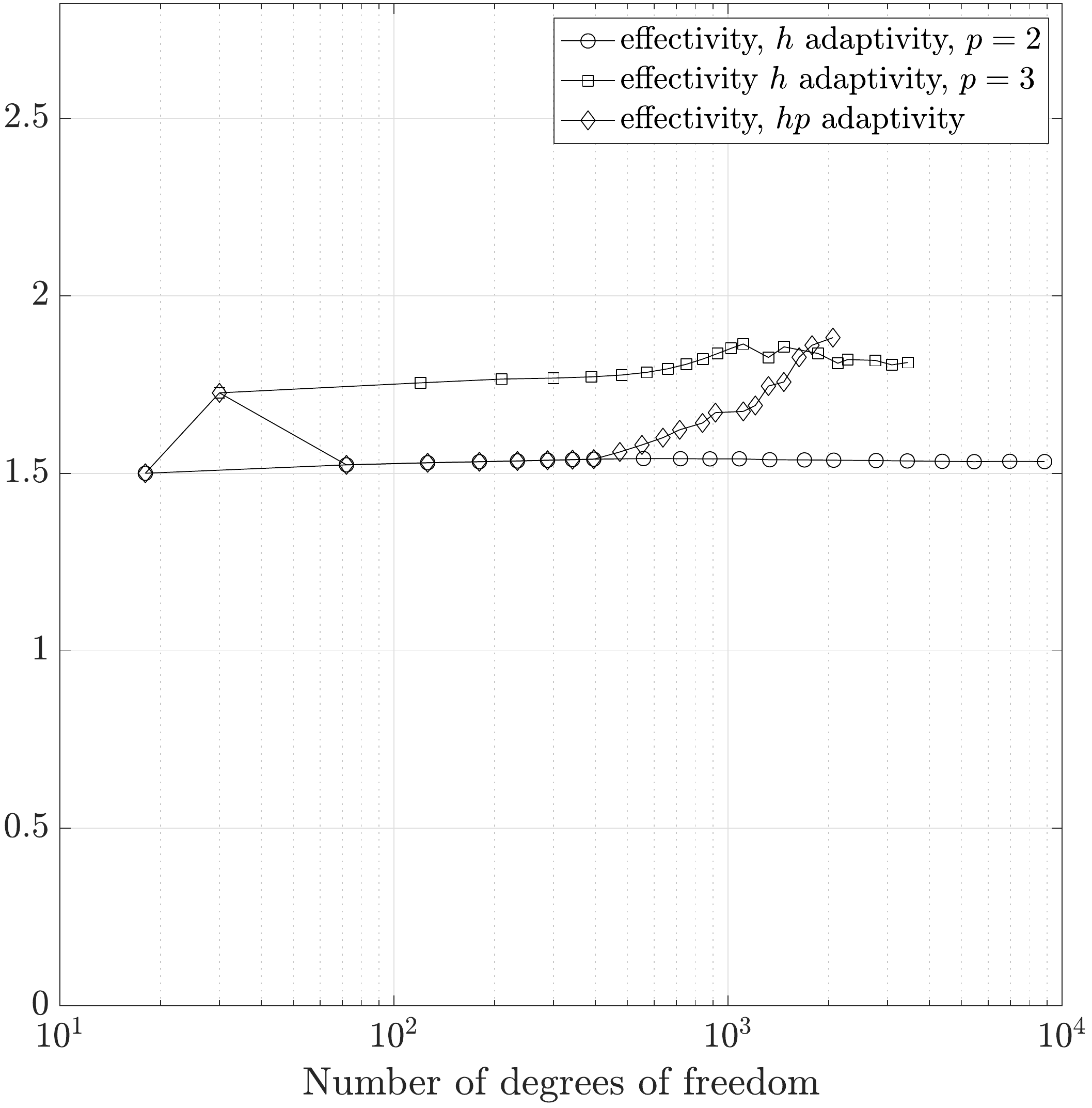}
	\caption{Errors and estimators (left) and effectivity indices (right) for the exact solution $u_1$ in~\eqref{u1} using the $\h$-adaptive algorithm with $p=2$ and $3$, and the $\h\p$-adaptive algorithm.}
	\label{figure:nr-u2}
\end{figure}

%%%
\begin{figure}
	\centering
	\includegraphics [angle=0, width=0.5\textwidth]{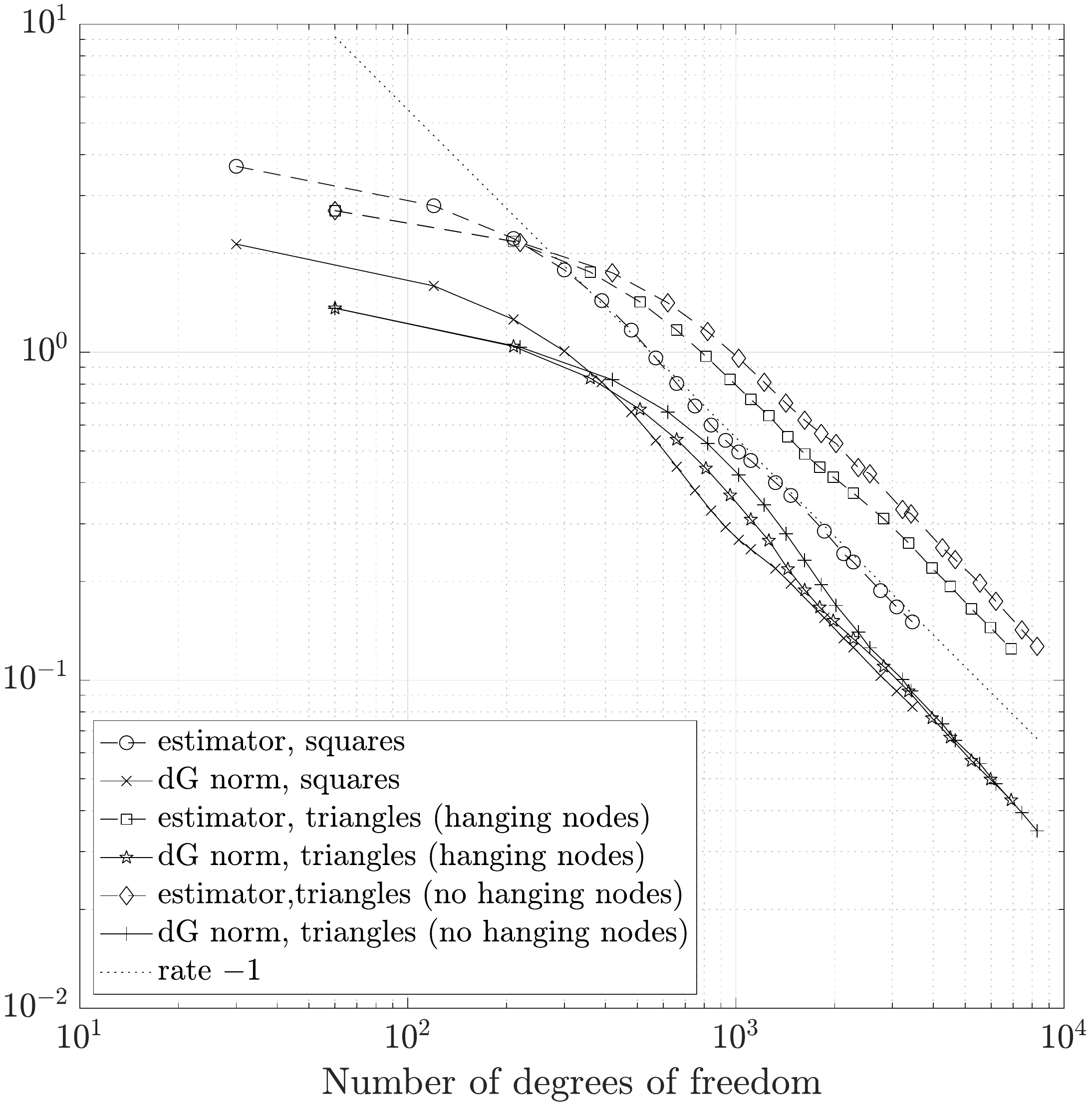}
	\includegraphics [angle=0, width=0.49\textwidth]{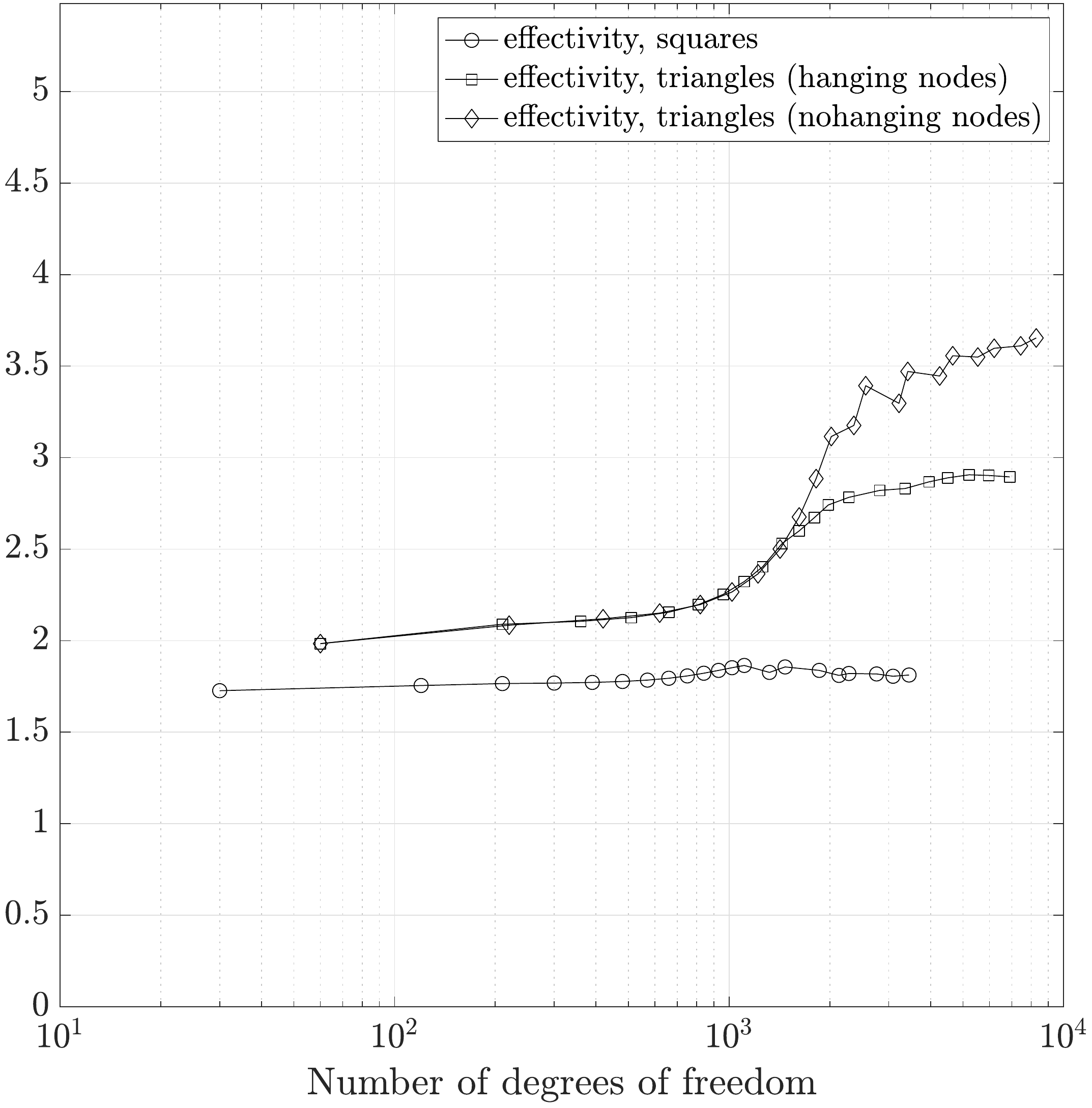}
	\caption{Errors and estimators (left) and effectivity indices (right) for the exact solution $u_1$ in~\eqref{u1} under $\h$-refinement with $p=3$
using triangular meshes with and without hanging nodes, and square meshes with hanging nodes.}
	\label{figure:nr-u2-2}
\end{figure}

%%%
\begin{figure}
	\centering
	\includegraphics [angle=0, width=0.49\textwidth]{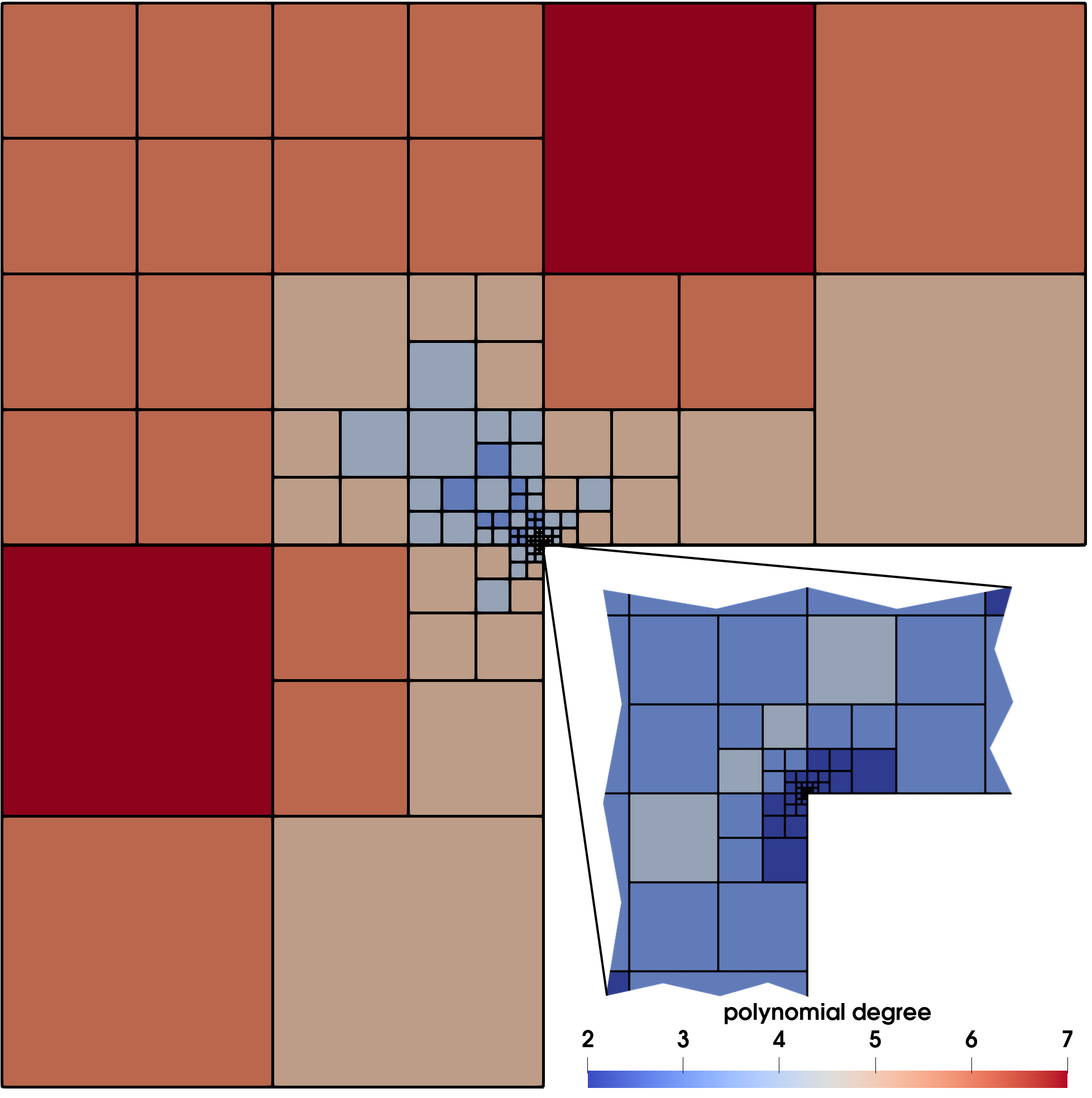}
	\includegraphics [angle=0, width=0.49\textwidth]{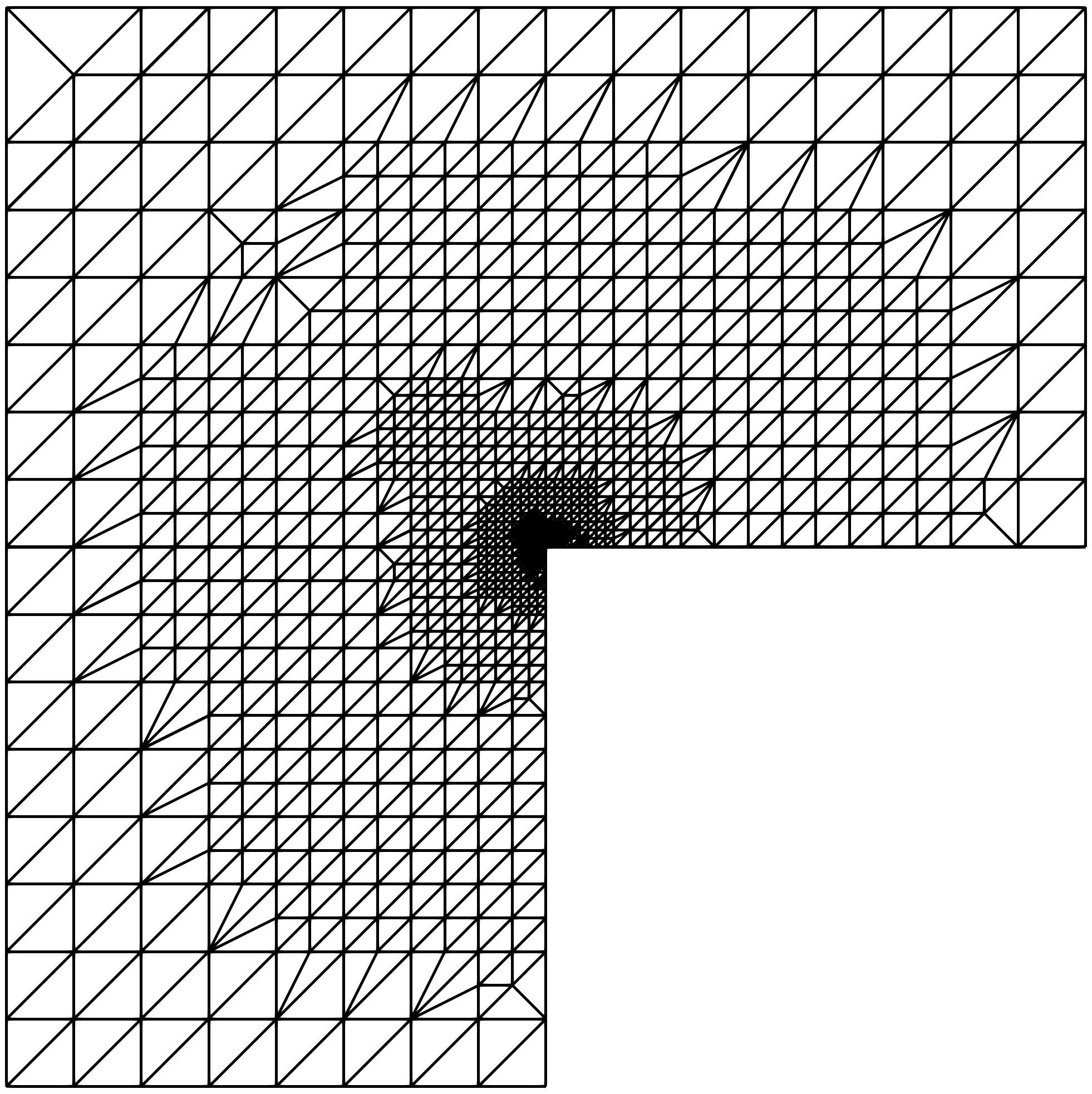}
\caption{Adapted meshes for solution $u_1$. Left: square mesh and polynomial degree distribution after 30 $\h\p$-adaptive steps. Inset: an enlargement of the region $[-0.0001, 0.0001]^2$ around the re-entrant corner.
Right: triangular mesh with hanging nodes removed, after 20 $h$-adaptive steps.}
	\label{figure:hp-meshes}
\end{figure}
	
\subsection{Adaptivity for the 3D benchmark solution $u_2$ in~\eqref{u2}}

Both adaptive algorithms were further applied to the smooth 3D benchmark problem, and the estimators, errors, and effectivities are plotted in Figure~\ref{fig:sin3d}.
In all cases, the initial mesh consisted of 64 cubic elements.
The $\h$-adaptive algorithm, using $p=2$ and $p=3$, may be seen to converge at the expected optimal rates of $-\frac{1}{3}$ and $-\frac{2}{3}$, respectively.
The effectivities also remain well behaved, settling down to values between $1.5$ and $3$ after an initial pre-asymptotic regime.

The $\h\p$-adaptive algorithm was initialised with~$\p=2$ and may be seen to converge significantly more rapidly than the $\h$-adaptive algorithms, with exponential convergence is observed for both the error an estimator.
In this case, the smoothness of the solution leads the $\h\p$-adaptive algorithm to always select $\p$-adaptivity, and no $\h$-refinement is performed.
As may be expected from Section~\ref{subsection:effectivity-index-p}, the effectivities grow throughout the simulation as the polynomial degree increases, reaching a value of 7.

%%%
\begin{figure}
\centering
\includegraphics [angle=0, width=0.495\textwidth]{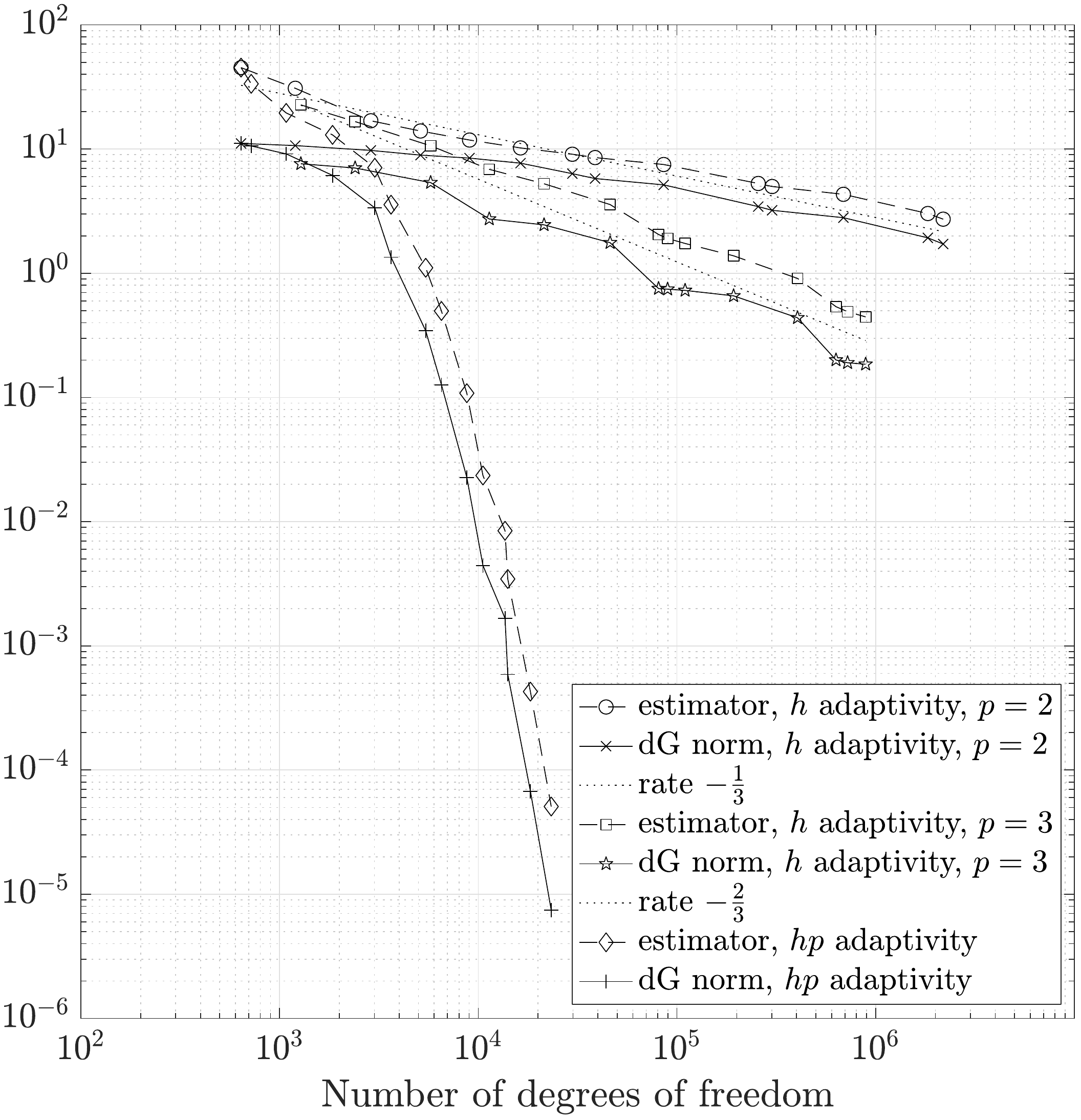}
\includegraphics [angle=0, width=0.49\textwidth]{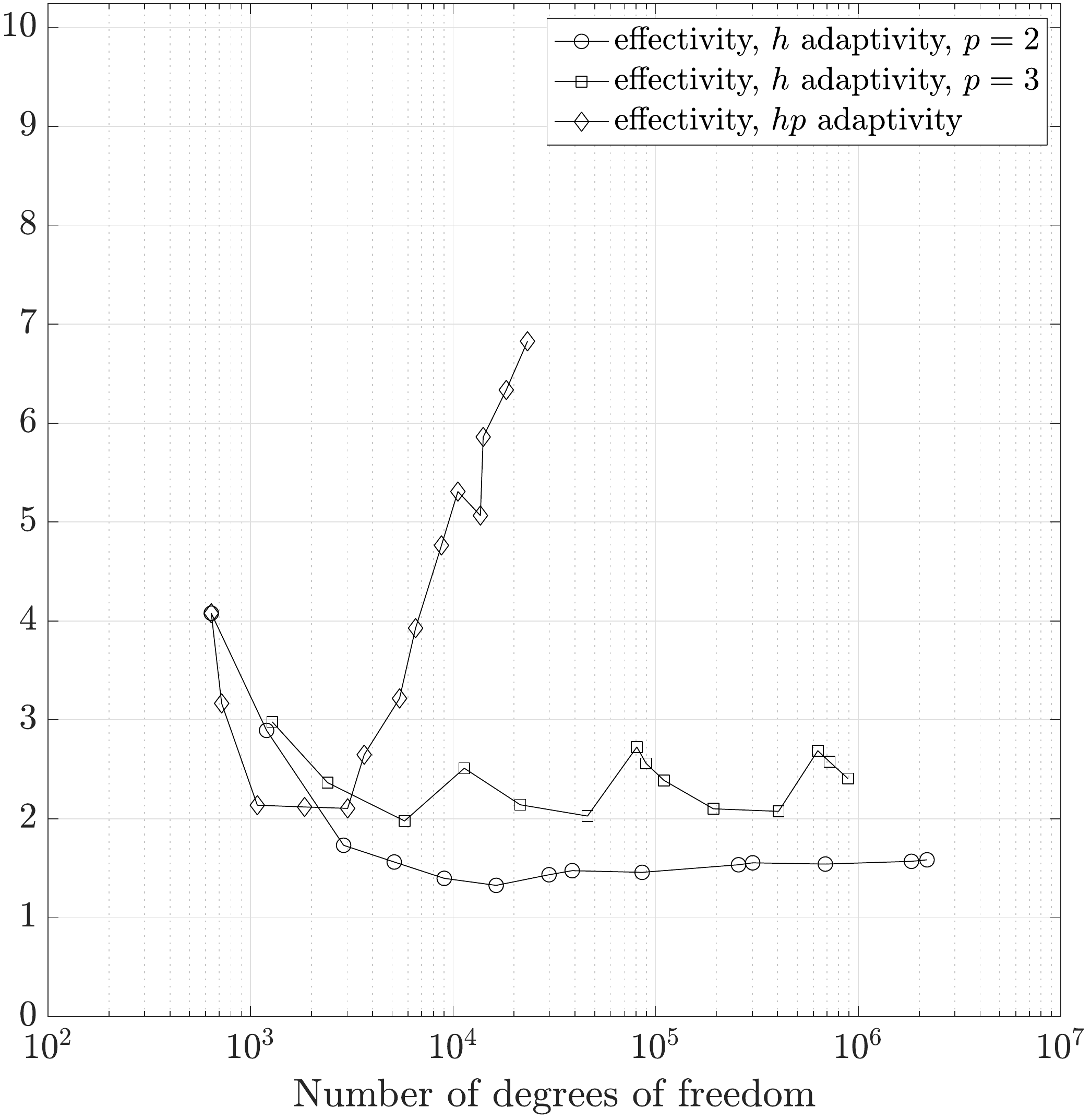}
\caption{Errors and estimators (left), and effectivity indices (right) for the 3D benchmark $u_2$ in~\eqref{u2} under $\h$- and $\h\p$-refinements using cube meshes with hanging nodes.}
\label{fig:sin3d}
\end{figure}
	
%%%%%%%%%%%%%%%%%%%%%%%%%%%%%%%%%%%%%%%%%%%%%%%%%%%%%%%%%%%%%%%%%%%%%%%%%%%
\section{Conclusions} \label{section:conclusions}
%%%%%%%%%%%%%%%%%%%%%%%%%%%%%%%%%%%%%%%%%%%%%%%%%%%%%%%%%%%%%%%%%%%%%%%%%%%
We have developed a residual-based error estimator for the $\h\p$-version interior penalty discontinuous Galerkin method for 2D and 3D biharmonic problems.
This is the first a posteriori error indicator for the dG methods with arbitrary polynomial order for the 3D biharmonic problem.
The upper and lower bounds are explicit in terms of the polynomial degree and, although the lower bound is suboptimal with respect to the polynomial degree, the dependence is algebraic and so exponential convergence is still attained.
Our analysis is based on an elliptic reconstruction of the dG solution combined with a generalised Helmholtz decomposition of the error.
We also discussed $\h\p$-explicit polynomial inverse estimates for bubble functions and extension operators.
The practical behaviour of the theoretical results was assessed through several numerical examples in two and three dimensions, using $\h$- and $\h\p$-adaptive algorithms.

In future, we plan to further investigate improving the dependence of the theoretical lower bounds on the polynomial degree.
Three dimensional generalisations of the $H^2$-stable extension operator in~\cite{lederer2018polynomial} could further lead to a polynomial robust a posteriori error estimator based on flux equilibration.
An additional avenue for future research is to design robust adaptive algorithms for singularly perturbed fourth order PDEs used e.g. in the modelling of two phase flows.
	
\section*{Acknowledgements}
This work was completed while ZD was working at the School of Mathematics at Cardiff University, UK;  the support from Cardiff University is gratefully acknowledged.
L. M. acknowledges the support of the Austrian Science Fund (FWF) project P33477.
O. J. S. acknowledges support from the EPSRC (grant number EP/R030707/1).
The authors would like to thank Professor P. Houston for kindly providing the AptoFEM software library, which was used to compute the numerical results, and Professor E. Georgoulis for his valuable advice.

	%-------------------------------------------------------------------------------------------------------------------------------------------
	{\footnotesize
		\bibliography{bibliogr}
	}
	\bibliographystyle{plain}
	
	%-------------------------------------------------------------------------------------------------------------------------------------------
\end{document}